\documentclass[11pt]{amsart}
\usepackage{
 amsmath, 
 amsxtra, 
 amsthm, 
 amssymb, 
 etex, 
 mathrsfs, 
 mathtools, 
 tikz-cd, 
 bbm,
 xr,
 comment,
 enumitem,
 enumerate}
\usepackage[toc]{appendix} 
\usepackage[all]{xy}
\usepackage{hyperref}
\usepackage{url}
\usepackage{tipa}
\usepackage{dsfont}
\usepackage{ textcomp }
\usepackage{stmaryrd} 
\usepackage{amssymb}
\usepackage{mathabx} 
\usepackage{amsmath} 
\usepackage{amscd}
\usepackage{amsbsy}
\usepackage{comment, enumerate}
\usepackage{color}
\usepackage{mathtools,caption}
\usepackage{tikz-cd}
\usepackage{longtable}
\usepackage[utf8]{inputenc}
\usepackage[OT2,T1]{fontenc}
\usepackage{changepage}
\usepackage{commath,enumerate}
\usepackage{enumitem}

\tikzcdset{scale cd/.style={every label/.append style={scale=#1},
    cells={nodes={scale=#1}}}}

\DeclareMathOperator{\Gal}{Gal}
\DeclareMathOperator{\rank}{rank}
\DeclareMathOperator{\ord}{ord}

\DeclareMathOperator{\Sel}{Sel}

\DeclareMathOperator{\Frac}{Frac}

\DeclareMathOperator{\GL}{GL}
\newcommand{\cyc}{{\mathrm{cyc}}}

\newtheorem{theorem}{Theorem}[section]
\newtheorem*{theorem*}{Theorem}
\newtheorem{lemma}[theorem]{Lemma}

\newtheorem{proposition}[theorem]{Proposition}

\newtheorem{corollary}[theorem]{Corollary}
\newtheorem{defn}[theorem]{Definition}
\newtheorem{defn-not}[theorem]{Notation and Definition}

\numberwithin{equation}{section}
\newtheorem{lthm}{Theorem} 

\newtheorem{remark}[theorem]{Remark}

\newcommand\EatDot[1]{}

\newcommand{\frakp}{{\mathfrak{p}}}
\newcommand{\lmto}{\mapsto}

\newcommand{\ac}{\mathrm{ac}}

\newcommand{\C}{\mathbb{C}}

\newcommand{\cE}{{\mathcal{E}}}
\newcommand{\cL}{{\mathcal{L}}}

\newcommand{\QQ}{\mathbb{Q}}
\newcommand{\g}{\mathbf{g}}
\newcommand{\ZZ}{\mathbb{Z}}
\newcommand{\Qp}{\mathbb{Q}_p}
\newcommand{\Zp}{\mathbb{Z}_p}

\definecolor{Green}{rgb}{0.0, 0.5, 0.0}
\newcommand{\bZ}{\mathbf{Z}}
\newcommand{\cO}{\mathcal{O}}

\newcommand{\Q}{\mathbb{Q}}
\newcommand{\F}{\mathbb{F}}
\newcommand{\Z}{\mathbb{Z}}

\newcommand{\Char}{\mathrm{Char}}

\newcommand{\is}{i\in\{1,2\}}
\newcommand{\BDP}{\mathrm{BDP}}

\newcommand{\fp}{\mathfrak{p}}
\newcommand{\fa}{\mathfrak{a}}

  \DeclareFontFamily{U}{wncy}{}
  \DeclareFontShape{U}{wncy}{m}{n}{<->wncyr10}{}
  \DeclareSymbolFont{mcy}{U}{wncy}{m}{n}
  \DeclareMathSymbol{\sha}{\mathord}{mcy}{"58}
  \DeclareMathSymbol{\zhe}{\mathord}{mcy}{"11}

\makeatletter
\newcommand{\mylabel}[2]{#2\def\@currentlabel{#2}\label{#1}}
\makeatother

\title[Iwasawa invariants of BDP Selmer groups and $p$-adic $L$-functions]{On the Iwasawa invariants of BDP Selmer groups and BDP $p$-adic $L$-functions}

\makeatletter
\let\@wraptoccontribs\wraptoccontribs
\makeatother

\usepackage{tikz,xcolor,hyperref}

\author[A.~Lei]{Antonio Lei}
\address[Lei]{Department of Mathematics and Statistics\\University of Ottawa\\
150 Louis-Pasteur Pvt\\
Ottawa, ON\\
Canada K1N 6N5}
\email{antonio.lei@uottawa.ca}

\author[K. Müller]{Katharina Müller}
\address[Müller]{D\'epartement de Math\'ematiques et de Statistique\\
Universit\'e Laval, Pavillion Alexandre-Vachon\\
1045 Avenue de la M\'edecine\\
Qu\'ebec, QC\\
Canada G1V 0A6}
\email{katharina.mueller.1@ulaval.ca}

\author[J. Xia]{Jiacheng Xia}
\address[Xia]{D\'epartement de Math\'ematiques et de Statistique\\
Universit\'e Laval, Pavillion Alexandre-Vachon\\
1045 Avenue de la M\'edecine\\
Qu\'ebec, QC\\
Canada G1V 0A6}
\email{jiacheng.xia.1@ulaval.ca}


\begin{document}
\begin{abstract}
Let $p$ be an odd prime. Let $f_1$ and $f_2$ be weight $2$ Hecke eigen cusp forms with isomorphic residual Galois representations at $p$. Greenberg--Vatsal and Emerton--Pollack--Weston showed that if $p$ is a good ordinary prime for the two forms, the Iwasawa invariants of their $p$-primary Selmer groups and $p$-adic $L$-functions over the cyclotomic $\mathbb{Z}_p$-extension of $\mathbb{Q}$ are closely related. The goal of this article is to generalize these results to the anticyclotomic setting. More precisely, let $K$ be an imaginary quadratic field where $p$ splits. Suppose that the generalized Heegner hypothesis holds with respect to both $(f_1,K)$ and $(f_2,K)$. We study relations between the Iwasawa invariants of the BDP Selmer groups and the BDP $p$-adic $L$-functions of $f_1$ and $f_2$. 
\end{abstract}

\maketitle

\section{Introduction}
\label{S: Intro}

Let $K$ be an imaginary quadratic field with discriminant $d_K$. Throughout this article, $p\ge3$  is a fixed prime number that is split in $K$. We fix once and for all an embedding $\overline{\QQ}\hookrightarrow\overline{\Qp}$ and let $\fp$ be the prime of $K$ above $p$ determined by this embedding given by the pre-image of the maximal ideal of the ring of integers of $\overline{\Qp}$. Let $f$ be a normalized newform (by which we always mean a Hecke eigenform) of $S_2(\Gamma_0(N),\cO)$, where $\cO$ is the ring of integers of a number field $L$.  Let $K_\infty$ be the anticyclotomic $\Zp$-extension of $K$. We assume that the class number of $K$ is coprime to $p$. In particular, this implies that the two primes above $p$ are totally ramified in $K_\infty$.

We assume that $N$ is coprime to $pd_K$ and we factorize $N$ into a product $N^+N^-$, where $N^+$ (resp. $N^-$) is the product of the prime factors of $N$ that are split (resp. inert) in $K$. We say that $f$ satisfies the \textit{generalized Heegner hypothesis} (with respect to $K$) if:
\begin{itemize}
    \item[\mylabel{GHH}{(GHH)}]$N^-$ is the square-free product of an even number of primes. 
\end{itemize}

In \cite{bertolinidarmonprasanna13,miljan}, under the assumption that $N^-=1$, Bertolini--Darmon--Prasanna and Brako\v{c}evi\'{c} constructed a $p$-adic $L$-function $L_p^\BDP(f)\in\Lambda^\ac$ whose square interpolates the $L$-values of $f_{/K}$ twisted by Hecke characters over $K$ of certain infinity type (this element is commonly known as the BDP $p$-adic $L$-function attached to $f$; For the definition of $\Lambda^\ac$, see \S\ref{prelim}). The construction of $L_p^\BDP(f)$ has been generalized by Hunter Brooks \cite{HB15} to the setting $N^->1$. The Iwasawa main conjecture predicts that the Pontryagin dual of the so-called BDP Selmer group $\Sel^\BDP(K_\infty,A_f)$ attached to $f$ over $K_\infty$  is torsion over $\Lambda^\ac$. Furthermore, its characteristic ideal (see Definition~\ref{def:invariants}) satisfies 
\begin{equation}\tag{IMC}\label{IMC}
    \left(L_p^\BDP(f)^2\right)\stackrel{?}{=}\Char_{\Lambda^\ac}\left(\Sel^\BDP(K_\infty,A_f)^\vee\otimes \Lambda^\ac\right)
\end{equation}
as ideals in $\Lambda^\ac$.

The conjecture \eqref{IMC} is equivalent to Perrin-Riou's Heegner point main conjecture formulated in \cite{perrinriou87} when $f$ is $p$-ordinary (see for example \cite[Theorem~5.2]{BCK21}). For similar results in the non-ordinary case, see \cite[Theorem~3.7]{CCSS} and \cite[Theorem~6.8]{castellawan1607}. This equivalence, together with the works of Howard \cite{howard04,howardcompositio1} (in the ordinary case) and the work of Kobayashi--Ota \cite{kobayashiota} combined with \cite{lei-zhao} (in the non-ordinary case), show that the inclusion $\subseteq$ in \eqref{IMC} holds under certain hypotheses. Under different hypotheses, the other  inclusion  or even the full equality of \eqref{IMC} has been studied in \cite[Theorem~5.2]{castellawan}, \cite[Theorem~B]{BCK21}, \cite[Theorem~6.8]{castellawan1607} and \cite[Theorem~5.8]{CCSS}.
In a recent work of Sweeting \cite[Theorem~E]{naomi}, \eqref{IMC} has been shown to hold in the $p$-ordinary case under very mild hypotheses.

The goal of this article is to study how Iwasawa invariants of the objects featured in \eqref{IMC} behave under congruences, generalizing results of Greenberg--Vatsal \cite{greenberg-vatsal} on objects defined over the cyclotomic $\Zp$-extension of $\QQ$ in the $p$-ordinary case (see \cite{kim09,hatleylei2019} for similar results in the $p$-non-ordinary case) and covering both the ordinary and non-ordinary cases simultaneously. Note that analogous results in the anticyclotomic setting have  been studied in \cite{kim17,CKL} under the hypothesis that $N^-$ is a squarefree product of an odd number of primes. In particular, \ref{GHH} is violated and the results therein concern different Selmer groups and $p$-adic $L$-functions from those featured in \eqref{IMC}. Under mild technical hypothesis one implication of the correpsonding main conjecture is proved in \cite{BertoliniDarmon2005}.

In order to explain our results, we first introduce certain notation that will be utilized throughout the  article.
For $\is$, let  $f_i\in S_2(\Gamma_0(N_i),\cO)$ be a newform. Here, $\cO$ is the ring of integers of a  number field $L$. We factor $N_i=N_i^+N_i^-$ as above and assume that \ref{GHH} holds for $f_i$.

Suppose that $p\nmid N_1N_2$ and fix a place $v$ of $\cO$ above $p$. We denote the completion of $\cO$ at $v$ by $\cO_v$ and fix a uniformizer $\varpi$ of $\cO_v$. Let 
\[\rho_i\colon G_\Q\to \textup{Aut}(V_i)=\textup{GL}_2(L_v)\] be the $p$-adic representation attached to $f_i$ of \cite{deligne69}, such that the identity $\textup{Tr}(\rho_i(\textup{Frob}_l))=a_l(f_i)$ holds for every prime $l$ coprime to $N_i p$, where $\textup{Frob}_l$ is the arithmetic Frobenius at $l$. Note that the Hodge--Tate weights of $V_i$ are $0$ and $1$ (our convention is that the cyclotomic character has Hodge--Tate weight $1$). The $p$-adic Tate module of the abelian variety attached to $f_i$ determines a $G_\Q$-stable lattice $T_i\subset V_i$. We define $A_i=V_i/T_i$. By an abuse of notation, we  denote the representation
\[G_\Q \to \textup{Aut}(T_i)=\textup{GL}_2(\cO_v)\]
by the same symbol $\rho_i$. The residual representation of $\rho_i$ is denoted by $\bar\rho_i$.
We assume throughout that 
\begin{equation}
\bar\rho_1\simeq\bar\rho_2
    \label{cong}\tag{Cong}
\end{equation}
as $G_\Q$-representations.

For $\is$, we consider the following conditions:
\begin{itemize}
    \item[\mylabel{FG}{(FG)}]  $\Sel^\BDP(K_\infty,A_i)^\vee$ is a finitely generated $\cO_v$-module.
    \item[\mylabel{H0}{($H^0$)}] $H^0(K_w,A_i)=\{0\}$ for all $w\mid p N^-_1N^-_2$ and $\is$.
\end{itemize}

Our main results are the following theorems.

\begin{lthm}[{Corollary~\ref{cor-equality-of-invariants}}]\label{thmA}
    Suppose that hypotheses \ref{H0} holds for both $f_1$ and $f_2$. Then \ref{FG}  holds for $i=1$ if and only if it holds for $i=2$. Furthermore, when these equivalent conditions hold, we have 
\[\rank_{\cO_v}(\Sel^{\BDP}(K_\infty,A_1)^\vee)+\sum_{\ell\in \Sigma_0}c_{\ell,1}=\rank_{\cO_v}(\Sel^{\BDP}(K_\infty,A_2)^\vee)+\sum_{\ell\in \Sigma_0}c_{\ell,2},\]
where $\Sigma_0$ is the set of primes dividing $N^+_1N^+_2$ and $c_{\ell,i}$ are explicit local terms defined in Lemma~\ref{lemma:lambda-invariants}.\end{lthm}

This generalizes similar results in \cite{HL2,HL21} where $N_1^-$ and $N_2^-$ are assumed to be $1$. Under appropriate hypotheses, \cite[Theorem~A]{lei-zhao} and \cite[Theorem~5.1]{HL21} show that \ref{FG} holds. Theorem \ref{thmA} allows us to illustrate new instances where \ref{FG} holds that are not covered by the aforementioned results (see Lemma~\ref{lem:example}).

On the analytic results, we prove a similar result under different hypotheses. In particular, we consider for $\is$:

\begin{itemize}
    \item[\mylabel{mu}{(mu)}]  $L_p^\BDP(f_i)$ is an Iwasawa function that is not divisible by $\varpi$.
    \item[\mylabel{div}{(div)}] If $N_i^-\neq 1$, then $p\nmid\prod_{q\mid N_i}(q-1)q(q+1)$.
    \item[\mylabel{sq-free}{(sq-fr)}] If $N_i^-\neq 1$, then $N_i$ is square-free.
\end{itemize}
\begin{remark}
    Assume that $\overline{\rho}$ is irreducible and let $M$ be the conductor of $\overline{\rho}$. We can write $N=M\prod_{i=1}^sq_i^{e_i}$ for some primes $q_i$ and exponents $e_i$. If \ref{div} is satisfied, then \cite[Theorem 1]{Diamond-Taylor} implies that $(q_i,M)=1$ and $e_i=1$ for $1\le i\le s$. Thus, \ref{sq-free} is satisfied if $M$ is square-free.
\end{remark}
\begin{lthm}[{Theorem~\ref{main-theorem-analytic}}]\label{thmB}
    Suppose  that the hypotheses  \ref{div}, \ref{sq-free} and \eqref{unit}\footnote{The hypothesis \eqref{unit} concerns the $p$-divisibility of certain  explicit factors. We refer the reader to the statement of Theorem~\ref{BDP-l-function-for-eigenforms} in the main body of the article for further details.} hold for both $f_1$ and $f_2$. Let  $\alpha(f_i)$ be the factor introduced in Theorem~\ref{thm:BDP}.\footnote{The constant $\alpha(f_i)$ is a quotient of Petersson products when $N_i^->1$. If $N_i^-=1$, then $\alpha(f_i)=1$.}
    Then, $\alpha(f_i)$ is a $p$-adic unit and \ref{mu}  holds for $i=1$ if and only if the same is true for $i=2$. Furthermore, when these equivalent conditions hold, the $\lambda$-invariants (see Definition~\ref{def:invariants}) of the BDP $p$-adic $L$-functions of $f_1$ and $f_2$ satisfy
    \[
    2\lambda(L_p^\BDP(f_1))+\sum_{q|\tilde N_3^{(1)}}d_{q,1}=2\lambda(L_p^\BDP(f_2))+\sum_{q|\tilde N_3^{(2)}}d_{q,2},
    \]
    where $\tilde N_3^{(1)}$, $\tilde N_3^{(2)}$, $d_{q,1}$ and $d_{q,2}$ are explicitly defined in Theorem~\ref{BDP-l-function-for-eigenforms}.
\end{lthm}

Note that under certain hypotheses, Hsieh \cite[Theorem~B]{hsiehnonvanishing} and Burungale \cite[Theorem~B]{burungale2} have shown that \ref{mu} holds for a large family of modular forms.  Analogous to the algebraic side, Theorem~\ref{thmB} allows us to find new examples of modular forms for which \ref{mu} holds (see Lemma~\ref{lem:example2}). The residual representations of these modular forms at $p$ are not absolutely irreducible. In particular, they are not covered by the previous results of Hsieh and Burungale.

Our proof of Theorem~\ref{thmB} allows us to deduce that, up to certain explicit Euler factors, the logarithms of the Heegner points on two abelian varieties of $\GL_2$-type satisfying \eqref{cong} are congruent modulo $\varpi$. In particular, this allows us to generalize \cite[Theorem~1.16]{KL} (with $m=1$), relaxing the strong Heegner hypothesis to \ref{GHH}. See Theorem~\ref{thm:KL} for the precise statement.

Our proof of Theorem~\ref{thmB} relies on a comparison between $L_p^\BDP(f_i)$  and the anticyclotomic specialization of the 2-variable Rankin--Selberg $p$-adic $L$-function multiplied by Katz' $p$-adic $L$-function. This is the origin of the factor $\alpha(f_1)$. Such comparison was already explored extensively in several previous works including \cite{CR,castella13,castellawan1607}. The crux of the method we employ dwells on congruence properties enjoyed by  Rankin--Selberg $p$-adic $L$-functions. As far as the authors are aware, such congruence properties are currently not proven directly for the BDP $p$-adic $L$-functions, which are constructed using quaternionic modular forms via the Jacquet--Langlands correspondence. Under certain hypotheses, results of Emerton \cite{emerton02} imply that these quaternionic modular forms are  congruent when \eqref{cong} holds. However, it seems to be an open problem in the generality we currently work with  (we refer the reader to the discussion in \cite[P.1780]{martin17} for more details). 

Finally, when the hypotheses of Theorems~\ref{thmA} and \ref{thmB} hold simultaneously, we can compare the local terms  to deduce:
\begin{lthm}[{Theorem~\ref{main}}]\label{thmC}
Let $f_1$ and $f_2$ be newforms of level $N_1$ and $N_2$ respectively, satisfying \ref{GHH}, \eqref{cong}, \ref{H0}, \ref{div}, \ref{sq-free} and \eqref{unit}. Assume that \ref{mu} holds for $f_1$ and that $\alpha(f_1)$ is a $p$-adic unit. Suppose that one inclusion in \eqref{IMC} holds for $f_2$ and that the full equality of \eqref{IMC}  holds for $f_1$. Then the full equality of \eqref{IMC} holds for $f_2$.

\end{lthm}

\section*{Acknowledgement}
Parts of these works were carried out during AL's visit at University College Dublin in fall 2022 supported by a Distinguished Visiting Professorship and the Seed Funding Scheme. He thanks UCD for the financial support and the warm hospitality. He also thanks Kazim Buyukboduk and Daniele Casazza for interesting discussions on subjects related to topics studied in this paper during his visit. The authors would like to thank Tobias Berger, Francesc Castella, Antonio Cauchi, Daniel Delbourgo, Neil Dummigan, Jeffrey Hatley, Ernest Hunter Brooks, David Loeffler, Kimball Martin, Ariel Pacetti and Jan Vonk for answering their questions during the preparation of the article. The authors would also like to thank Chan-Ho Kim, Chao Li and Luochen Zhao for their helpful suggestions and comments. Finally, we thank the anonymous referee for carefully reading an earlier version of the article as well as constructive feedback. All three authors’ research is supported by the NSERC Discovery Grants Program RGPIN-2020-04259 and RGPAS-2020-00096.

\section{Notation}
\label{prelim}

Let $\mathfrak{c}$ be an integral ideal of $K$. We denote the ray class field of conductor $\mathfrak{c}$ by $K_{\mathfrak{c}}$. Let $\mathcal{K}=\bigcup_{n\ge1} K_{p^n}$ and denote $\Gal(\mathcal{K}/K)$ by $H_{p^\infty}$. Analogously, we define $H_{\mathfrak{p}^\infty}$ as the Galois group of $\bigcup_{n\ge1}K_{\mathfrak{p}^n}/K$. Note that $H_{p^\infty}\cong \Z_p^2\times \Delta$ for a finite abelian group $\Delta$. Our assumption that $p$ is coprime to the class number of $K$ implies that $\Delta$ is of order coprime to $p$. 

The group $\Gal(K/\Q)$ acts on the abelian group $H_{p^\infty}$ via the formula $\sigma \mapsto\tilde c\sigma \tilde c^{-1}$, where $\tilde c$ is a lift of $c \in \Gal(K/\Q)$ to $\Gal(\mathcal{K}/\QQ)$. We will frequently use the notation $\sigma^c$ to denote $\tilde c\sigma \tilde c^{-1}$. Let $H^+_{p^\infty}$ be the subgroup on which $\Gal(K/\Q)$ acts trivially and define $H^-_{p^\infty}=H_{p^\infty}/H^+_{p^\infty}$. Let $K^\textup{a}_\infty\subset \mathcal{K}$ be the unique subextension with Galois group $H^-_{p^{\infty}}$ (note that $K^{\textup{a}}_\infty$ contains the anticyclotomic $\Z_p$-extension $K_\infty$ of $K$). We denote $\Gal(K_\infty/K)$ by $\Gamma^\ac$. 

Let $L_v$ and $\cO_v$ be defined as in the introduction.
Let $R$ be the completion of the integral closure of $\cO_v$ in the maximal unramified extension of $L_v$. 

We define the following Iwasawa algebras
\begin{align*}
&  \Lambda_p=R[[H_{p^\infty}]],\quad\Lambda_{\mathfrak{p}}=R[[H_{\mathfrak{p}^\infty}]],\quad\Lambda^\textup{a}=R[[H^-_{p^\infty}]],\\
 &\Lambda^\textup{ac}=R[[\Gamma^{\ac}]],\quad\Lambda^\ac_{\cO_v}=\cO_v[[\Gamma^\ac]].    
\end{align*}

There are natural projections
\[\pi^\textup{a}\colon \Lambda_p \to \Lambda^\textup{a}, \quad \pi^\textup{ac}\colon \Lambda_p\to \Lambda^\textup{ac}.\]
These projections extend naturally to projections on the corresponding fraction fields. By an abuse of notation, we  denote these projections by the same symbols.

Given an element $\sigma$ in $H^-_{p^\infty}$, we denote any lift of $\sigma$ to $H_{p^\infty}$ by $\tilde{\sigma}$. Note that the element $\tilde{\sigma}^{1-c}\in \Lambda_p$ does not depend on the lift $\tilde{\sigma}$ we chose by the definition of $H^+_{p^\infty}$. Thus, we obtain a well-defined map
\[\mathcal{V}\colon \Lambda^\textup{a}\to \Lambda_p.\] 
such that 
\[\pi^\textup{a}\circ \mathcal{V}(\sigma)=\sigma^2.\]

In the remainder of the article, we frequently employ the notion of Iwasawa invariants of elements in  $\Lambda^\ac$ and $\Lambda^\ac_{\cO_v}$. Note that $\cO_v$ and $R$ are local rings with  maximal ideals $\mathfrak{m}_{\cO_v}$ and $\mathfrak{m}_R$, respectively. For each of the rings $R$ and $\cO_v$ we fix a generator of the maximal ideal, denoted by $\varpi_{\cO_v}$ and $\varpi_R$, respectively. If the ring is clear from the context we omit the superscript.  

\begin{defn}\label{def:invariants}For every element $F\in \Lambda^\ac$,  the Weierstrass preparation theorem says that  
\[F=\varpi_R^m UG,\]
for  an integer $m\ge0$, a unit $U$ in $\Lambda^\ac$ and a distinguished polynomial $G$. We define the $\lambda$-invariant and $\mu$-invariant of $F$ by $\lambda(F)=\deg(G)$ and $\mu(F)=m$, respectively. If $F\in \Lambda^\ac_{\cO_v}$, we define  $\lambda(F)$ and $\mu(F)$ in the same way.\footnote{Note that the absolute ramification indices in $\cO_v$ and in $R$ are the same. In particular, the $\mu$- and $\lambda$-invariants are independent of the choice of Iwasawa algebra.}

Given a finitely generated torsion module $M$ over $\Lambda^\ac$, it is pseudo-isomorphic to 
\[
\bigoplus_{i=1}^r\Lambda^\ac/F_i,
\]
for some integer $r \geq 1$ and nonzero elements $F_i\in\Lambda^\ac$. We define the characteristic ideal of $M$ to be
\[
\Char_{\Lambda^\ac}(M)=\left(\prod_{i=1} ^rF_i\right)\subset \Lambda^\ac.
\]
We define $\mu(M)$ and $\lambda(M)$ to be $\mu(F)$ and $\lambda(F)$, respectively, where $F$ is any generator of   $\Char_{\Lambda^\ac}(M)$.
We  define these objects in a similar fashion if  $\Lambda^\ac$ is replaced by $\Lambda^\ac_{\cO_v}$.
\end{defn}

\section{Iwasawa invariants of BDP Selmer groups under congruences}
We study how the Iwasawa invariants of BDP Selmer groups behave for congruent modular forms, proving Theorem~\ref{thmA} stated in the introduction.
Throughout, we fix $f_i\in S_2(\Gamma_0(N_i),\cO)$ for $\is$ satisfying the conditions \ref{GHH}, \eqref{cong} and \ref{H0}.
Note that the hypothesis \ref{H0} implies that $H^0(K,A_i)=H^0(K_\infty,A_i)=0$.
\begin{lemma} 
\label{lemma:kidwell-condition}
Let $w$ be a prime of $K$ dividing $N^-_1N^-_2$. Let $I_w$ be the inertia subgroup inside $G_{K_w}$. Then $H^1(G_{K_w}/I_w,A_i^{I_w})$ and $H^0(K_w,T_i)$ are trivial.
\end{lemma}

\begin{proof}
Let $K_w^\cyc$ be the cyclotomic $\Zp$-extension of $K_w$ and write $G_w^\cyc=\Gal\left(\overline{K_w}/K_w^\cyc\right)$. Since $w\nmid p$, we have the inclusion $I_w\subset G_w^\cyc$ and that  
$$G_{K_w}/G_w^\cyc\cong \Gal(K_w^\textup{cyc}/K_w)\cong (G_{K_w}/I_w)^{\textup{pro-p}}\cong\Zp.$$ 
Therefore, the triviality of $H^1(G_{K_w}/I_w,A_i^{I_w})$ would follow from  showing that $$H^1\left(\Gal(K_w^\cyc/K_w),A_i^{G^\cyc_w}\right)=\{0\}.$$
Indeed, under the hypothesis \ref{H0}, we have $H^0(K_w,A_i)=\{0\}$, which implies that $H^0(K_w^\cyc,A_i)=\{0\}$. Thus, the first assertion of the lemma follows.

For the second assertion, note that $A_i\cong T_i\otimes \QQ_p/\ZZ_p$. Thus, $H^0(K_w,A_i)=\{0\}$ implies that the same holds for $H^0(K_w,T_i)$. 
\end{proof}

\begin{defn}\label{def:BDPSel} Let $\Sigma$ be the set of all primes dividing $N_1N_2p$ and let $\Sigma_0$ be the set of primes dividing $N^+_1N^+_2$. 
We define the BDP Selmer group of $f_i$ over $K_n$ by
\[
\Sel^\BDP(K_n,A_i)=\ker\left(H^1(K_\Sigma/K_n,A_i)\to \prod_{w|\Sigma\setminus\{\fp^c\}}H^1(K_{n,w},A_i) \right)
\]
and the  imprimitive BDP Selmer group (with respect to $\Sigma_0$) 
\[\Sel^\BDP_{\Sigma_0}(K_n,A_i)=\ker\left(H^1(K_\Sigma/K_n,A_i)\to \prod_{w|\Sigma\setminus (\Sigma_0\cup\{\fp^c\})}H^1(K_{n,w},A_i)
\right).\]
(Here, given a set $S$ of primes, the product  $\displaystyle \prod_{w|S}$ runs through all the primes dividing an element of $S$.)
\end{defn}
Note that $\Sel^{\BDP}(K_n,A_i)\subset \Sel^{\BDP}_{\Sigma_0}(K_n,A_i)$ and that any prime in $\Sigma_0$ is finitely decomposed in $K_\infty/K$ (see \cite[Theorem~2]{brink}). In the standard definition of BDP Selmer groups, one usually requires the cocycles to be unramified at primes away from $p$ instead of trivial (see \cite[(1.1)]{kobayashiota} or \cite[Definition~2.2]{castellaJLMS}). Our assumption that $H^0(K_w,A_i)=\{0\}$ for all $w\mid N^-_1N^-_2$ implies that the natural map
\[H^1(K_w,A_i)\to H^1(I_w,A_i)\]
is injective (see also Lemma \ref{lemma:kidwell-condition}). It then follows from \cite[the proof of Proposition 3.1]{kidwell} that our definition of BDP Selmer groups coincides with the standard one. We adopt the above definition to  make the arguments on the algebraic side more accessible to the readers.

We now prove several preliminary lemmas regarding these Selmer groups.
\begin{lemma}
\label{lemma:surj} Assume that $\Sel^\BDP(K_\infty,A_i)$ is $\Lambda^\textup{ac}_{\cO_v}$-cotorsion.
Then the global-to-local map
\[H^1(K_\Sigma/K_\infty,A_i)\to \prod_{w|\Sigma\setminus\{\mathfrak{p}^c\}} H^1(K_{\infty,w},A_i)\] is surjective.
\end{lemma}
\begin{proof}
We follow the strategy of \cite[proof of Proposition~A.2]{pollack_weston_2011}. As in loc. cit., since the local cohomology groups $ H^1(K_{\infty,w},A_i)$ do not contain nontrivial $\Lambda^\ac_{\cO_v}$-submodules of finite index, it suffices to show that the cokernel of the global-to-local map is finite. 

Fix $\is$. For an integer $t$, let $B_t$ denote the tensor product of $G_K$-representations $A_i\otimes \kappa^t$, where $\kappa$ is a fixed character \[\kappa\colon \Gamma^\ac\to 1+p\Z_p\] that sends a topological generator of $\Gamma^\ac$ to $1+p$. Note that $B_t\cong A$ as $\Gal(K_\Sigma/K_\infty)$-modules and thus  we may replace $A_i$ by $B_t$  in the statement of the  lemma.

It follows from Lemma \ref{lemma:kidwell-condition} and \cite[Propositions 4.2 and  4.3]{kidwell} that $H^1(K_{\infty,w},B_t)$ is $\Lambda_{\cO_v}^\ac$-cotorsion for all $w|\Sigma\setminus\{p\}$. Furthermore, by assumption,  $\Sel^\BDP(K_\infty,A_i)$ is also $\Lambda^\textup{ac}_{\cO_v}$-cotorsion.
Therefore, we may choose  $t$ such that 
\[\rank_{\cO_v}(\Sel^{\BDP}(K_n,B_t)^\vee)=\rank_{\mathcal{O}_{v}}(H^1(K_{n,w},B_t)^\vee)=0\] 
for all $n\ge0$ and all $w| \Sigma\setminus\{p\}$. When we need to specify at which level the place $w$ is defined, we will write $w_n$ and $w_\infty$, respectively.

If $w|\fp$, \cite[Proposition 1]{greenberg89} tells us that $H^1(K_{\infty,v_\infty},B_t)^\vee$ is of rank 2 over  $\Lambda^\textup{ac}_{\cO_v}$ for all $t$. Consequently, for all but finitely many choices of $t$, the $\cO_v$-module $H^1(K_{n,v_n},B_t)^\vee$ is of rank $2p^n$ for all $n \geq 0$.  Furthermore, \cite[Proposition 3]{greenberg89} implies that $H^1(K_\Sigma/K_n,B_t)^\vee$ has $\cO_v$-rank at least $2p^n$.\footnote{While it is assumed that $\cO_v=\Z_p$  in loc. cit., the proof carries over to the more general case in a straightforward manner.} Thus, we may choose $t$ such that the cokernel of the global-to-local map
\[
H^1(K_\Sigma/K_n,B_t)\to \prod_{w| \Sigma\setminus\{\mathfrak{p}^c\}} H^1(K_{n,w},B_t)
\] 
is finite. It remains to show that this finite cokernel is uniformly bounded independently of $n$, which can be proved via the same argument given in \cite[proof of Proposition (2.1)]{greenberg-vatsal}. We give an outline of the proof for the convenience of the reader. 

Let $T_{-t}=T_i\otimes\kappa^{-t}$. Note that $\textup{Hom}(B_t,L_v/\cO_v(1))=T_{-t}$. Consider the following global-to-local maps
\begin{align*}
    \alpha_0 \colon &H^1(K_\Sigma/K_n,B_t)\to \prod_{w| \Sigma\setminus\{\mathfrak{p}^c\}}H^1(K_{n,w},B_t),\\
    \alpha\colon &H^1(K_\Sigma/K_n,B_t)\to \prod_{w| \Sigma} H^1(K_{n,w},B_t),\\
    \beta\colon &H^1(K_\Sigma/K_n,T_{-t})\to \prod_{w| \Sigma} H^1(K_{n,w},T_{-t}).
\end{align*}
By an abuse of notation, we denote the prime of $K_n$ lying above $\fp^c$ by the same symbol.
Let \[G_{B_t}=\langle\textup{im}(\alpha),H^1(K_{n,\fp^c},B_t)\rangle\] be the smallest subgroup of $\prod_{w|\Sigma}H^1(K_{n,w},B_t)$ containing the image of $\alpha$ and $H^1(K_{n,\fp^c},B_t)$. Let $G_M\subset \prod_{w| \Sigma} H^1(K_{n,w},T_{-t})$ be the orthogonal complement of $G_{B_t}$ under local Tate duality. 

We have chosen $t$ such that  the cokernel of $\alpha_0$ is finite. For such $t$, the group $G_M$ is also finite. 
 Note that $\ker(\beta)^\vee$ is a subgroup of $H^2(K_\infty/K_n,B_t)$, which is trivial (by the assumption that $p$ is odd). Let $S_M$ be the preimage of $G_M$ under $\beta$. It follows that $S_M$ is finite and that its cardinality is bounded by $H^1(K_\Sigma/K_n,T_{-t})_{\textup{tors}}$. A standard argument  (see \cite[page 19]{greenberg-vatsal}) then shows that 
\[\left\vert H^1(K_\Sigma/K_n,T_{-t})_{\textup{tors}}\right\vert \le\left \vert H^0(K_\infty,A_i)\right\vert ,\]
which is trivial under \ref{H0} (see the discussion before Lemma~\ref{lemma:kidwell-condition}). This shows that $\alpha_0$ is surjective. In particular, the cardinality of the cokernel of $\alpha_0$ is indeed bounded independently of $n$ as required.
\end{proof}
\begin{lemma}
\label{no-finite-submodule} If the $\Lambda_{\cO_v}^\ac$-module $\Sel^\BDP_{\Sigma_0}(K_\infty,A_i)^\vee$ is torsion, then it does not contain a non-trivial finite submodule.
\end{lemma}
\begin{proof}
Let $B_t$ and $T_{-t}$ be defined as in the proof of Lemma~\ref{lemma:surj}. Once again, we may replace $A_i$ by $B_t$. We choose $t$ such that 
\[\rank_{\cO_v}(\Sel^{\BDP}_{\Sigma_0}(K_n,B_t)^\vee)=0,\quad\text{and}\quad \rank_{\cO_v}(H^1(K_\Sigma/K_n,B_t)^\vee)=2p^n\] for all $n\ge0$. Note that \[\rank_{\Lambda^\textup{ac}_{\cO_v}}(H^1(K_\Sigma/K_\infty,B_t)^\vee)=2.\] It follows from \cite[Proposition 3]{greenberg89} that 
\[\rank_{\Lambda^\textup{ac}_{\cO_v}}(H^2(K_\Sigma/K_\infty,B_t)^\vee)=0.\]
By \cite[Proposition 4 and 5]{greenberg89}, $H^1(K_\Sigma/K_\infty,B_t)^\vee$ has no nontrivial finite submodule. We can now conclude as in \cite[Proof of proposition 4.14]{Greenberg} that $H^1(K_\Sigma/K_\infty,B_t)_{\Gamma^\textup{ac}}$ is trivial. 

Take $n=0$ in the proof of Lemma \ref{lemma:surj}. Recall that 
\[
H^1(K_\Sigma/K,B_t)\to \left(\prod_{w|\Sigma\setminus(\Sigma_0\cup\{\mathfrak{p}^c\})} H^1(K_{w},B_t)\right)
\] 
is surjective.  As $H^0(K_\infty,B_t)=H^0(K_\infty,A)=0$, it follows from the inflation-restriction exact sequence that $H^1(K_\Sigma/K_\infty,B_t)^{\Gamma^{\textup{ac}}}\cong H^1(K_\Sigma/K,B_t)$. We thus obtain a surjection
\[
H^1(K_\Sigma/K,B_t)^{\Gamma^\textup{ac}}\to \left(\prod_{v\in \Sigma\setminus(\Sigma_0\cup\{\mathfrak{p}^c\})} H^1(K_{\infty,w},B_t)\right)^{\Gamma^\textup{ac}}.
\]

The short exact sequence
\[0\to \Sel_{\Sigma_0}^\BDP(K_\infty,B_t)\to H^1(K_\Sigma/K_\infty,B_t)\to \left(\prod_{w|\Sigma\setminus(\Sigma_0\cup\{\mathfrak{p}^c\})} H^1(K_{\infty,w},B_t)\right)\to 0\]
gives the exact sequence 
\begin{align*}H^1(K_\Sigma/K,B_t)^{\Gamma^\textup{ac}}&\to \left(\prod_{w|\Sigma\setminus(\Sigma_0\cup\{\mathfrak{p}^c\})}H^1(K_{\infty,w},B_t)\right)^{\Gamma^\textup{ac}}\to \\&\to \Sel_{\Sigma_0}^\BDP(K_\infty,B_t)_{\Gamma^\textup{ac}}\to H^1(K_\Sigma/K_\infty,B_t)_{\Gamma^\textup{ac}}.\end{align*}
As the first map is surjective and the last term is trivial, we deduce that
\[\Sel_{\Sigma_0}^\BDP(K_\infty,B_t)_{\Gamma^\textup{ac}}=0,\]
which implies that $\Sel^\BDP_{\Sigma_0}(K_\infty,B_t)^\vee$ does not contain a nontrivial finite submodule.
\end{proof}

We now relate the Iwasawa invariants of the standard BDP Selmer groups to their imprimitive counterparts.

\begin{lemma}
\label{lemma:lambda-invariants}
For any prime $l$ in $K$ that lies above a rational prime $q$ that is split in $K/\Q$, we define the polynomial $P_l^{f_i}(X)=\det(1-X\textup{Frob}_l|(V_i)_{I_l})$, where $I_l$ is the inertia group at $l$. Let $$\mathcal{P}^{f_i}_l=P^{f_i}_l(N(l)^{-1}\gamma_l)\in \Lambda^\ac_{\cO_v},$$  where $\gamma_l$ is the Frobenius of $l$ in $K_\infty/K$ and $N$ denotes the norm map from $K$ to $\Q$.

Then $\Sel^\BDP_{\Sigma_0}(K_\infty,A_i)$ is cotorsion over $\Lambda^\ac_{\cO_v}$ if and only if $\Sel^\BDP(K_\infty,A_i)$ is cotorsion over $\Lambda^\ac_{\cO_v}$. Furthermore, when these Selmer groups are indeed cotorsion over $\Lambda^\ac_{\cO_v}$, we have 
\begin{align*}
    \mu(\Sel^\BDP_{\Sigma_0}(K_\infty,A_i)^\vee)&=\mu(\Sel^\BDP(K_\infty,A_i)^\vee),\\
    \lambda(\Sel^{\BDP}_{\Sigma_0}(K_\infty,A_i)^\vee)&=\lambda(\Sel^{\BDP}(K_\infty,A_i)^\vee)+\sum_{l| \Sigma_0}\lambda(\mathcal{P}^{f_i}_l).
\end{align*}
\end{lemma}
\begin{remark}\label{rk:local-terms}
It can be verified that $\mathcal{P}^{f_i}_l$ has trivial $\mu$-invariant and that it is a unit if $l\mid q$ for a rational prime $q$ such that $a_q(f_i)=0$.
\end{remark}
\begin{proof}
By definition, we have an exact sequence
\begin{equation}\label{eq:SEQ-BDP}
    0\to \Sel^{\BDP}(K_\infty, A_i)\to \Sel^{\BDP}_{\Sigma_0}(K_\infty,A_i)\to \prod_{w|\Sigma_0}H^1(K_{\infty,w},A_i).
\end{equation}
Every prime dividing $\Sigma_0$ is finitely decomposed in $K_\infty$. Thus, $H^1(K_{\infty,w},A_i)$ is cotorsion over $\Lambda_{\cO_v}^\ac$ (see \cite[proof of Proposition (2.4)]{greenberg-vatsal}). This proves the equivalence of the cotorsionness of the two Selmer groups.

Suppose that these Selmer groups are indeed cotorsion. Then  Lemma~\ref{lemma:surj} tells us that the last arrow in \eqref{eq:SEQ-BDP} is surjective. As discussed in \cite[proof of Proposition (2.4)]{greenberg-vatsal}, $H^1(K_{\infty,w},A_i)^\vee$ is $\Lambda_{\cO_v}^\ac$-torsion, with characteristic ideal generated by $\mathcal{P}^{f_i}$. Furthermore, its $\mu$-invariant is zero. Therefore, the lemma follows from the additivity of Iwasawa invariants in exact sequences.
\end{proof}

\begin{lemma}
\label{lemma:isom} Let $\varpi$ be a uniformizer of $\cO_v$.
Assume that $A_1[\varpi]\cong A_2[\varpi]$ as $G_K$-modules.\footnote{Note that this is a weaker condition than \eqref{cong}.} Then 
\[\Sel^{\BDP}_{\Sigma_0}(K_\infty,A_1)[\varpi]\cong\Sel^{\BDP}_{\Sigma_0}(K_\infty,A_2)[\varpi].\]
\end{lemma}
\begin{proof}
Fix $\is$. Let $w$ be a prime of $K_\infty$ dividing $ \Sigma\setminus \Sigma_0$. Recall that $H^0(K_w,A_i)=\{0\}$. In particular, since $K_{\infty,w}/K_w$ is a pro-$p$ extension, $H^0(K_{\infty,w},A_i)=\{0\}$. On taking cohomology of the short exact sequence 
\[0\to A_i[\varpi]\to A_i\to A_i\to 0,\]
we deduce that
\[H^1(K_{\infty,w},A_i[\varpi])\cong H^1(K_{\infty,w},A_i)[\varpi].\]
We also have
\[H^1(K_\infty,A_i[\varpi])\cong H^1(K_\infty,A_i)[\varpi].\]
Thus, 
\[\Sel^{\BDP}_{\Sigma_0}(K_\infty,A_i)[\varpi]\cong \Sel^{\BDP}_{\Sigma_0}(K_\infty,A_i[\varpi]),\]
where $\Sel^\BDP_{\Sigma_0}(K_\infty,A_i[\varpi])$ is defined in the obvious manner.

The  assumption on the isomorphism between $A_1[\varpi]$ and $A_2[\varpi]$ implies that
\[\Sel^{\BDP}_{\Sigma_0}(K_\infty,A_1[\varpi])\cong \Sel^{\BDP}_{\Sigma_0}(K_\infty,A_2[\varpi]).\]
Hence the result follows.
\end{proof}

We are now ready to prove Theorem~\ref{thmA}.
\begin{corollary}
\label{cor-equality-of-invariants}
Assume that $A_1[\varpi]\cong A_2[\varpi]$ as $G_K$-modules. Assume further that $\Sel^\BDP(K_\infty,A_1)$ is $\Lambda^\textup{ac}_{\cO_v}$-cotorsion and that $\mu(\Sel^{\BDP}(K_\infty,A_1)^\vee)=0$. 
Then $\Sel^{\BDP}(K_\infty,A_2)$ is cotorsion,  $\mu(\Sel^{\BDP}(K_\infty,A_2)^\vee)=0$ and 
\[\lambda(\Sel^{\BDP}(K_\infty,A_1)^\vee)+\sum_{l\in \Sigma_0}\lambda(\mathcal{P}^{f_1}_{l})=\lambda(\Sel^{\BDP}(K_\infty,A_2)^\vee)+\sum_{l\in \Sigma_0}\lambda(\mathcal{P}^{f_2}_{l}).\]
\end{corollary}
\begin{proof}
Note that $\Sel_{\Sigma_0}^{\BDP}(K_\infty,A_1)[\varpi]$ is finite due to the assumption that $\Sel_{\Sigma_0}^{\BDP}(K_\infty,A_1)$ is $\Lambda^\textup{ac}_{\cO_v}$-cotorsion and that its $\mu$-invariant vanishes.
By the proof of Lemma~\ref{lemma:isom},
\[
\Sel_{\Sigma_0}^{\BDP}(K_\infty,A_1)[\varpi]\cong \Sel_{\Sigma_0}^{\BDP}(K_\infty,A_2)[\varpi].
\]
In particular, $\Sel_{\Sigma_0}^{\BDP}(K_\infty,A_2)[\varpi]$ is also finite. This implies that the $\Lambda_{\cO_v}^\ac$-module $\Sel^{\BDP}(K_\infty,A_2)^\vee$ is torsion  with vanishing  $\mu$-invariant. Furthermore, \[\lambda(\Sel_{\Sigma_0}^{\BDP}(K_\infty,A_i)^\vee)=\ord_\varpi(\vert \Sel_{\Sigma_0}^{\BDP}(K_\infty,A_i)[\varpi]\vert)\] by Lemma \ref{no-finite-submodule}. Hence, the asserted equation involving $\lambda$-invariants  follows from  Lemmas \ref{lemma:lambda-invariants} and \ref{lemma:isom}.
\end{proof}
\subsection{An application}
We conclude this section by showing that Theorem~\ref{thmA} implies the existence  of an infinite family of modular forms  whose BDP Selmer groups over $K_\infty$ have vanishing  $\mu$-invariant. 
\begin{defn}
    Let $f=\sum a_n q^n\in S_2(\Gamma_0(N),\overline{\Z})$. Let $N=N_1N_2N_0$ be a decomposition on $N$ into pairwise coprime integers, such that $N_1N_2$ is square-free. Let $w$ be a place above $p$ in the coefficient ring of $f$. Let $\phi$ and $\psi$ be Dirichlet characters. We say that $f$ has partial Eisenstein decent by $(\phi,\psi,N_1,N_2,N_0)$ if
    \begin{itemize}
        \item $a_l\equiv \phi(l)+l\psi(l) \pmod w   \quad\forall (l,N)=1$
        \item $a_l\equiv \phi(l) \pmod w\quad\forall l\mid N_1$
        \item $a_l\equiv l\psi(l)\pmod w\quad \forall l\mid N_2$
        \item $a_l\equiv 0\pmod w\quad \forall l\mid N_0$.
    \end{itemize}
\end{defn}
\begin{lemma}\label{lem:example}
    Let $f$ be a newform of level $N$ and weight $2$ and let $\chi$ be a quadratic character such that $\chi(p)=-1$. Assume that $p\in\{3,5\}$ and that $f$ has partial Eisenstein decent by $(\chi,\chi,N_1,N_2,N_0)$. Let $K$ be an imaginary quadratic field satisfying \ref{GHH} with respect to $N\textup{cond}(\chi)$. We assume that $N^-=1$ and that $11$ and $19$ split in $K$. Then 
    \[\mu((\Sel^\BDP(K_\infty, A)^\vee)=0.\]
\end{lemma}
\begin{proof}
    Let $E/\Q$ be the elliptic curve with LMFDB label 11a.2 if $p=5$ and the elliptic curve with LMFDB label 19.a2 if $p=3$. Then 
    $$E[p]\cong\F_p\oplus \F_p(\omega)$$ as $G_\Q$-representations. Let $g$ be the modular form associated to the quadratic twist $E^{\chi}$. Recall that $v$ is a fixed place of $K$ above $p$. Note that $H^0(K_v,E^\chi[p^\infty])=\{0\}$ by construction. As discussed in \cite[Remark 32]{kriz}, we have $a_n(f)\equiv a_n(g)\pmod \varpi$ for all $n$ that are coprime to $pN\textup{cond}(E^{\chi})$. By \cite[Theorem 34]{kriz}, $$\overline{\rho}_f\cong\F_p(\chi)\oplus \F_p(\chi\omega)$$ as $G_\QQ$-representations. We recall from \cite[theorem 1.5.1]{CGLS} that the $\mu$-invariant of $\Sel^\BDP(K_\infty,E^\chi[p^\infty])^\vee$ vanishes. Thus, the lemma follows from Corollary \ref{cor-equality-of-invariants}.
\end{proof}
\begin{remark}
    Let $M/\Q$ be the extension cut out by $\chi$. If $q$ is a rational prime coprime to $p$ that ramifies in $M$, then $\chi\omega$ is ramified at $q$ as well and $q^2\mid N$. Recall that $N_1N_2$ is square free. Thus, $\chi$ is unramified at all primes dividing $N_1N_2$. Assume now that $q$ is inert in $K$ and let $w_K$ and $w_M$ be places above $q$ in $K$ and $M$, respectively. Then $M_{w_M}\subset K_{w_K}$. Thus, $H^0(K_{w_K},A)\neq \{0\}$. So, we can only apply Corollary \ref{cor-equality-of-invariants} in the case $N^-=1$.
\end{remark}

\section{$\BDP$ $p$-adic $L$-functions of congruent modular forms}
The main goal of this section is to prove Theorem~\ref{thmB}. As discussed in the introduction, our strategy is to relate the  BDP $p$-adic $L$-function to the anticyclotomic specialization of the product of the Katz $p$-adic $L$-function and a 2-variable Rankin--Selberg $p$-adic $L$-function. We review these objects in the following subsection.

\subsection{Review on $p$-adic $L$-functions}
 Let $\psi$ be a Hecke character of $K$ of conductor $\mathfrak{f}$. We say that $\psi$ is of infinity type $(k,j)$ for integers $k$ and $j$ if 
 \[\psi((\alpha))=\alpha^k\overline{\alpha}^j\quad \textup{for all }\alpha \equiv 1 \pmod {\mathfrak{f}}.\]
 If $\mathfrak{f}=(1)$, then $\psi((\alpha))=\alpha^k\overline{\alpha}^j$ for all principal ideals. Note that a Hecke character of trivial conductor and infinity type $(k,j)$ can only exist if $\vert \mathcal{O}_K^\times \vert$ divides $k-j$.
 We say that $\psi$ is anticyclotomic if it factors through $H_{p^\infty}^-$. Note that such a character is of infinity type $(k,-k)$.  Let $\psi$ be an anticyclotomic character of trivial conductor. Let $q$ be a rational prime that is is inert in $K/\Q$ and write $\mathfrak{q}=(q)$. Then we have
 \[\psi(\mathfrak{q})=q^k\left(\overline{q}\right)^{-k}=1.\]
 
The $L$-function associated to a Hecke character $\psi$ over $K$ is given by
\[L(\psi,s)=\sum \frac{\psi(\mathfrak{a})}{N(\mathfrak{a})^s}\]
for $\Re(s)\gg0$, where the sum runs over all integral ideals of $K$ that are coprime to $\mathfrak{f}$. We have an Euler product
\[L(\psi,s)=\prod_{\frakp} \left(1-\frac{\psi(\mathfrak{p})}{N(\mathfrak{p})^s}\right)^{-1},\]
where the product runs over all prime ideals coprime to $\mathfrak{f}$.

\begin{theorem}
 There exist CM periods $\Omega\in \C^\times$ and $\Omega_p\in R^\times$ and an element $L_{K,p}\in \Lambda$ such that
\[L_{K,p}({\psi}^{-1})=\Gamma(k)\Omega^{j-k}\Omega_p^{k-j}(\sqrt{\vert D\vert})^j(2\pi)^{-j} (1-{\psi}(\overline{\mathfrak{p}}))  (1-p^{-1}{{\psi}}(\mathfrak{p}))L({\psi}^{-1},0)\]for all Hecke characters of infinity type $(k,j)$ with $0\le -j<k$ and trivial conductor.
\end{theorem}
\begin{proof}
This is proved in \cite[Section 5.3.0]{katz78}.
\end{proof}
\begin{theorem}\label{thm:BDP}
Let $\mathcal{N}^+$ be an ideal of $\cO_K$  such that $\mathcal{O}_K/\mathcal{N}^+\cong \Z/N^+\Z$.
    There exists a $p$-adic L-function $L^\BDP_p(f)\in \Lambda^\ac$ such that for each anticyclotomic Hecke character $\psi$ of trivial conductor  and infinity type $(-m,m)$, where $m\equiv 0\ \pmod{p-1}$, we have
    \begin{align*}L^\BDP_p(f)(\psi)^2=& \left(\frac{\Omega_p}{\Omega_k}\right)^{4m }\frac{\Gamma(m)\Gamma(m+1)\psi({\mathfrak{N}^+})^{-1}}{4(2\pi)^{1-2m}\sqrt{D_K}^{2m-1}}\\
    & \times\alpha(f)(1-a_p\psi(\overline{\fp})p^{-1}-\psi(\overline{\fp})^2p^{-1})^2L(f,\chi,1),\end{align*}
    where $\alpha(f)\in\cO_v$ is defined by
    \[\alpha(f) =\begin{cases} \frac{\langle f, f\rangle }{\langle f^{\mathrm{JL}}, f^{\mathrm{JL}}\rangle} &N^->1,\\
    1&\text{otherwise.}
    \end{cases}\] Here, $\langle-,-\rangle$ denotes the Petersson inner product and $f^{\mathrm{JL}}$ denotes Jacquet--Langlands transform of $f$, which are normalized as in \cite[Sections 1-2]{prasanna2006}.
\end{theorem}
\begin{proof}
    The construction of $L_p^\BDP(f)$ originates from \cite{bertolinidarmonprasanna13} and \cite{miljan}. The normalization we employ here is the one given in \cite[Section 4]{BCK21}. See \cite[Proposition 2.1]{castellawan1607} for the interpolation formula. Note that in loc. cit. it is assumed that $\cO_v=\Z_p$. But the $p$-adic Waldspurger formula used in the proof holds in our current setting. The same formula can also be deduced from \cite[Theorem 5.6]{burungale2} if $\vert \mathcal{O}_K^\times \vert =2$.\footnote{Note that \cite{castellawan1607} and \cite{burungale2} use opposite conventions for infinity types of Hecke characters.}
\end{proof}

Let $N$ be an integer that is coprime to $p$. Let $S_2(\Gamma_0(N),\cO_v)$  be the set of cusp forms of level $\Gamma_0(N)$ with coefficients in $\cO_v$. Set
\[
\g=\sum_{(\fa,\fp)=1}[\fa]q^{N(\fa)}\in\Zp[[H_{\fp^\infty}]][[q]]
\]
to be a CM Hida family over $\Zp[[H_{\fp^\infty}]]$.

Given a character of $H_{p^\infty}$, we may factorize it in a unique way $\varphi\psi$, where $\varphi$ factors through $H_{\fp^\infty}$ and $\psi$ factors through the norm map $H_{p^\infty}\rightarrow \Zp^\times$. Following \cite[\S6.3]{LLZ2}, we define:
\begin{defn}
Let    $\cL_{N,p}:S_2(\Gamma_0(N),\cO_v)\rightarrow \Frac \Lambda$ be the $\cO_v$-linear map satisfying
\[
\cL_{N,p}(h)(\varphi\psi)=\frac{\langle \g(\varphi)^*,\Xi_{1/N|d_K|}(h,\varphi\psi)^{\ord,p}\rangle_{N\vert d_K\vert}}{\langle \g(\varphi),\g(\varphi)\rangle_{ \vert d_K\vert}},
\]
where $\langle-,-\rangle_{M}$ denotes the Petersson product at level $\Gamma_1(M)\cap \Gamma_0(p)$, 
\[
\Xi_{1/N|d_K|}(h,\varphi\psi)^{\ord,p}=e_{\ord}[\cE_{1/N|d_K|}(\psi-1,-1-\varphi_\QQ-\psi)\cdot h].
\]
Here, $\g(\varphi)^*$ denotes the complex conjugate of the specialization of $\g$ at $\varphi$, $e_{\ord}$ is Hida's ordinary projector, $\cE_{1/N|d_K|}(-,-)$ denotes the $p$-depleted Eisenstein series over $\mathrm{Spec}\ \Zp[[\Zp^\times]]^2$ given in \cite[Definition~5.3.1]{LLZ1}  and $\varphi_\QQ$ denotes the composition of $\varphi$ with $\Zp^\times\hookrightarrow (\mathcal{\cO}_K\times \Zp)^\times\rightarrow H_{p^\infty}$. Note that we have written the characters additively. 
\end{defn}
When $h\in S_2(\Gamma_0(N),\cO_v)$ is a normalized Hecke eigenform, $\cL_{N,p}(h)$ is a 2-variable Rankin--Selberg $p$-adic $L$-function attached to $h$, interpolating complex $L$-values of the base change of $h$ to $K$ twisted by Hecke characters of certain infinity type (see the discussion in \cite[end of \S6.3]{LLZ2}).

\begin{remark}
\label{rem:linearity}
    Recall that the Petersson inner product is linear in the first coordinate and skew linear in the second. We have chosen $L$ to be a number field containing the coefficients of the new forms $f_1$ and $f_2$. In particular, we can assume that $L$ is totally real. Using the interpolating property of the 2-variable Rankinn--Selberg $p$-adic $L$-function we see that for $f_1,f_2$ in $S_2(\Gamma_0(N),\cO_v)$ and $a\in \cO_v$
    \[\mathcal{L}(h_1-h_2)=\mathcal{L}(h_1)-\mathcal{L}(h_2),\quad \mathcal{L}(ah)=a\mathcal{L}(h).\]
\end{remark}

\subsection{Anticyclotomic specialization}
In this subsection, we study the anticyclotomic specialization of $L_{K,p}$ and $\cL_{N,p}$. Throughout, we fix a normalized Hecke eigenform $f\in S_2(\Gamma_0(N),\cO_v)$, where $N$ is coprime to $p\cdot d_K$. Note that $f$ may be an oldform at level $N$. We factorize $N=N^+N^-$ and assume that \ref{GHH} holds as before. 

\begin{defn}
\item[i)]Let $\iota\colon \Lambda^\ac\to \Lambda^\ac$ be the $R$-linear map sending an element $\gamma\in\Gamma^\ac$ by $\gamma\lmto \gamma^2$. Let $L^{\textup{ac}}_{K,p}\in \Lambda^\textup{ac}$ be the image of $L_{K,p}$ under the composition $\iota\circ \pi^\textup{ac}$.
\item[ii)] Let $\mathcal{L}^{\textup{ac}}_{N,p}(f)\in \Frac{\Lambda^{\textup{ac}}}$ be the image of $\mathcal{L}_{N,p}(f)$ under $\pi^{\textup{ac}}$.
\end{defn}

\begin{theorem}
\label{thm:integrality-initial-case}
Assume that $f$ is a newform of level $N$ satisfying \ref{div} and \ref{sq-free}. 
Then there exists a unit  $u_K\in (\Lambda^{\ac})^\times$ (only depending on $K$)
such that $$\mathcal{L}_{p,N}^\textup{ac}(f)L^\textup{ac}_{K,p}=\alpha(f)Nu_K\sigma^{-1}_{\mathcal{N}^+}L_p^\BDP(f)^2\in \Lambda^\textup{ac},$$ where $\sigma_{\mathcal{N}^+}$ is the image of $\mathcal{N}^+$ under the Artin homomorphism and $\mathcal{N}^+$ is given as in the statement of Theorem~\ref{thm:BDP}.
\end{theorem}
\begin{proof} Upon writing down the interpolation formulae alluded to in the proof of 
\cite[Proposition 2.7]{castellawan1607}, we deduce that $\mathcal{L}_{p,N}^\textup{ac}(f)L^\textup{ac}_{K,p}$ is equal to, up to a unit of $\Lambda^\ac$ of the form $Nu_K\sigma^{-1}_{\mathcal{N}^+}$ (where $u_K$ only depends on $K$), the $p$-adic $L$-function $\alpha(f)L_p^\BDP(f)^2$. By assumption \ref{div} and \cite[Theorem 2.4]{prasanna2006}, $\alpha(f)$ is a $p$-adic integer. As $L_p^\BDP(f)$  lies inside $\Lambda^\ac$, the theorem follows. 
\end{proof}
\begin{remark}
    If $f\in S_2(\Gamma_0(N),\Z)$ is associated to a rational elliptic curve, then $\alpha(f)$ is a $p$-adic unit whenever $p$ is non-Eisenstein and  the residual representation $\overline{\rho_{f}}\colon G_\Q\to GL_2(\cO_v/\varpi)$ attached to $f$ is ramified at each prime $l\mid N_i^-$; see \cite[page 912]{prasanna2006}.
\end{remark}

\begin{defn}
\label{def:fcheck}
     Let $f=\sum b_nq^n$ be a normalized newform of level $N'$ and $N$ an integer such that $N'\mid N$. Assume that $N'$ satisfies \ref{sq-free}. If $q\mid (N/N')$, we assume that $q^2\mid N$. Let $\check{f}=\sum a_n q^n$ be the eigenform of level $N$ with Fourier coefficients $a_q=0$ for $q\mid N/N'$ and $a_q=b_q$ for all $q$ that are coprime to $N/N'$. 
\end{defn}
While $\check f$ depends on $N$, we suppress it from the notation since it will be clear from the context what $N$ is. Recall that we decompose any Hecke character of $H_{p^\infty}$ as a product $\varphi\psi$, where $\varphi$ factors through $H_{\mathfrak{p^\infty}}$. We denote by $\g(\varphi)$ the specialization of the Hida family $\g$ at $\varphi$. Let $P_q(\g(\varphi),f,q^{-s})$ be the Euler factor of the tensor product of the representations associated to $\g(\varphi)$ and $f$. The following lemma follows directly from definition.
\begin{lemma}\label{behavior-split-primes}
    Let $\tilde{N}_3$ be a product of rational primes that split in $K/\Q$. Let $\varphi\psi$ be an anticyclotomic Hecke character of infinity type $(b,-b)$ with $b \equiv 0 \pmod {(p-1)}$ and let $\rho$ be the  character on $\Gamma^{ac}$ induced from $\varphi\psi$. Then 
    \[\prod_{q\mid \tilde{N}_3}\prod_{\mathfrak{q}\mid q}\mathcal{P}_{\mathfrak{q}}^f(\rho)=\prod_{q\mid \tilde{N}_3}P_q(\g(\varphi),f,q^{-(1+b)}).\]
\end{lemma}
The following theorem can be regarded as an analytic analogue of Lemma~\ref{lemma:lambda-invariants}.

\begin{theorem}
\label{BDP-l-function-for-eigenforms} 
Let $N'|N$ be positive integers that are coprime to $p\cdot d_K$. Assume that $N'$ satisfies \ref{sq-free}. If $q$ is a prime dividing $N/N'$,  we assume further that $q^2\mid N$.
Let $f=\sum b_qq^n\in S_2(\Gamma_0(N'),\cO_v)$ be a normalized newform such that $(f,N')$ satisfy the hypotheses \ref{GHH} and \ref{div}. Let $\check{f}$ be the eigenform of level $N$ defined in Definition \ref{def:fcheck}.
Then 

\begin{itemize}
    \item[(a)]$\mathcal{L}^\textup{ac}_{p,N}(\check{f})L^\textup{ac}_{K,p}\in \Lambda^\textup{ac}$.
\item[(b)]Let $\tilde{N}_1$ be the product of all primes $q$ that divide $N$ but not $N'$ and are inert in $K/\Q$. Let $\tilde{N}_2$ be the product of all common prime divisors of $N'$ and $N/N'$ that are inert in $K/\Q$. Let $\tilde{N}_3$ be the product of all $q\mid N/N'$ that are split in $K/\Q$. If
\begin{equation}
    \prod_{q\mid \tilde{N}_1}(1+q-b_q)(1+q+b_q)\prod_{q\mid \tilde{N}_2}(q-b_q)(q+b_q)\in\cO_v^\times,\tag{unit}\label{unit}
\end{equation} then 
\begin{align*}
    \mu(\mathcal{L}^\ac_{p,N}(\check{f})L^\textup{ac}_{K,p})&=\mu(\mathcal{L}^\ac_{p,N'}(f)L^\textup{ac}_{K,p}), \\
    \lambda(\mathcal{L}_{p,N}^\ac(\check{f})L^\textup{ac}_{K,p})&=\lambda(\mathcal{L}^\ac_{p,N'}(f)L^\textup{ac}_{K,p})+\sum_{\mathfrak{q}\mid\tilde{N}_3}\lambda(\mathcal{P}^f_\mathfrak{q}),
\end{align*}
where the sum runs over primes of $K$ dividing $\tilde{N}_3$ and $\mathcal{P}^f_\mathfrak{q}$ is defined as in Lemma~\ref{lemma:lambda-invariants}.
\end{itemize}
\end{theorem}

\begin{proof}
 Let $\varphi\psi$ be an anticyclotomic Hecke character of trivial conductor and infinity type $(b,-b)$ with $b\equiv 0 \pmod {(p-1)}$. Let $h\in\{f,\check{f}\}$. Let $M=N'\vert d_k\vert$ if $h=f$ and $M=N\vert d_k\vert$ if $h=\check{f}$. 
Let $\mathcal{D}({\g}(\varphi),h, 1/M,s)$ be the integral defined in \cite[\S4.2, P.55]{LLZ1}. Here, $s$ is a complex variable. 
It follows from \cite[Proposition 5.4.2]{LLZ1} that
\begin{equation}
\label{interpolation-llz}
    \mathcal{L}_{p,M}(h)(\varphi\psi)=(*) \mathcal{D}({\g}(\varphi),h, 1/M,1+b),
\end{equation}
where $(*)$ is a factor only depending on $f$ but not on $h$ nor  $M$.

We now compare $\mathcal{D}({\g}(\varphi),\check{f}, 1/N\vert d_k\vert,1+b)$ and $\mathcal{D}({\g}(\varphi),f, 1/N'\vert d_k\vert,1+b)$. Let 
\[C({\g}(\varphi),h,M,s)=\prod_{q\mid M}P_q(\g(\varphi),h,q^{-s})\sum_{S(M)}a_nb_nn^{-s},\]
where $S(M)$ is the set of integers whose  prime factors divide $M$, and $a_n$ and $b_n$ are the Fourier coefficients of ${\g}(\varphi)$ and $h$, respectively. We recall from \cite[Theorem 4.2.3]{LLZ1} that $C({\g}(\varphi),h,M,s)$ is a polynomial in $q^{-s}$ for $q\mid M$. In particular, it is an entire function in $s$.

Thanks to \cite[Theorem 4.2.3]{LLZ1}, $\mathcal{D}({\g}(\varphi),h, 1/M,s)$ and $C({\g}(\varphi),h,M,s)$ differ only by a $p$-adic unit and the completed $L$-function $\Lambda({\g}(\varphi),f,s)$, which is independent of the choice of $h$ (see \cite[Proposition 4.1.5]{LLZ1} for the precise definition).   In the remainder of the proof, we prove that up to a $p$-adic integer, $C({\g}(\varphi),\check{f},N\vert d_k\vert,s)$ and $C({\g}(\varphi),f,N'\vert d_k\vert,s)$ differ by $\prod_{q|\tilde{N}_3} P_q(\g(\varphi),f,q^{-s})$. It would  then follow from  Lemma \ref{behavior-split-primes} and \eqref{interpolation-llz} that $\mathcal{L}_{p,N}^\ac(\check{f})L^\textup{ac}_{K,p}$ and $\mathcal{L}^\ac_{p,N'}(f)L^\textup{ac}_{K,p}\prod_{\mathfrak{q}\mid\tilde{N}_3}\mathcal{P}^f_\mathfrak{q}$ have the same Iwasawa invariants. This would imply the assertions of the theorem since $\mathcal{L}_{p,N'}^\textup{ac}(f)L^\textup{ac}_{K,p}\in\Lambda^\ac$ by Theorem~\ref{thm:integrality-initial-case} and $\mathcal{P}_\mathfrak{q}^f\in\Lambda^\ac_{\cO_v}$ with $\mu(\mathcal{P}_\mathfrak{q}^f)=0$ for all $\mathfrak{q}|\tilde N_3$.

Without loss of generality, we may assume that $N/N'=q^l$ for some prime $q$. We will distinguish three cases:
 \begin{itemize}
     \item[(1)] $q\not\mid N'$
     \item[(2)] $q\mid N'$ but $q^2\not \mid N'$
     \item[(3)] $q^2\mid N'$
 \end{itemize}

 In  case (3), $\check{f}=f$ and $\mathcal{P}^f_\mathfrak{q}$ is a unit.

We consider case (1). By the definition of $\check{f}$, we have $b_q=0$. Thus,  
\[\sum_{S(N\vert d_k\vert)}a_nb_nn^{-s}=\sum_{S(N'\vert d_k\vert)}a_nb_nn^{-s},\]
which implies that
\[C({\g}(\varphi),\check{f},N\vert d_k\vert,s)=P_q(\g(\varphi),f,q^{-s})C({\g}(\varphi),f,N'\vert d_k\vert,s).\]

 If $q$ is inert in $K$, then $$P_q(\g(\varphi),f,q^{-(1+b)})=(1-b_q q^{-1}+q^{-1})(1+b_qq^{-1}+q^{-1})$$ does not depend on $\varphi$ (see \cite[Proposition 4.1.2]{LLZ1} for details). In particular, it does not depend on $\varphi\psi$ and is a $p$-adic integer. 
  
 It remains to consider case (2). In this case, $$\prod_{q\mid N}P_q(\g(\varphi),f,q^{-s})=\prod_{q\mid N'}P_q(\g(\varphi),f,q^{-s})$$ and only the sum $\sum_{S(M)}a_nb_nn^{-s}$ depends on $M$ (since the Fourier coefficient at $q$ differs). Decomposing $\sum_{S(M)}a_nb_nn^{-s}$ into a finite Euler product, we see that only  an Euler factor at $q$ is altered when passing from $f$ to $\check{f}$. 
 
 If $q$ is inert in $K$, this factor again does not depend on $\varphi\psi$. If $q$ splits in $K$, we can proceed as in case (1).
\end{proof}
\begin{remark}
    Keeping track of the units in the proof of Theorem~\ref{BDP-l-function-for-eigenforms}  yields 
    \begin{align*}L^\ac_{K,p}\mathcal{L}^\ac_{p,N}(\check{f})=(N/N') \prod_{q\mid \tilde{N}_1}(1+q^{-1}-b_qq^{-1})(1+q^{-1}+b_qq^{-1})\\
    \prod_{q\mid \tilde{N}_2}(1+b_qq^{-1})(1-b_qq^{-1}) \prod_{\mathfrak{q}\mid \tilde{N}_3}\mathcal{P}^{f}_{\mathfrak{q}}\times L^\ac_{K,p}\mathcal{L}^\ac_{p,N} (f)\end{align*}
\end{remark}
 As an immediate consequence, we deduce:
 \begin{corollary}
 \label{integrality-up-to-p} Let $h\in S_2(\Gamma_0(N),\cO_v)$ be an $L_v$-linear combination of Hecke eigenforms which satisfy the hypotheses of Theorem~\ref{BDP-l-function-for-eigenforms}.   Then,
 \[L^\textup{ac}_{K,p}\mathcal{L}_{p,N}^\textup{ac}(h)\in \Lambda^{\ac} \otimes \Q_p.\]
 \end{corollary}

\subsection{Integrality and anticyclotomic specialization}
The goal of this subsection is to prove an integral version of Corollary~\ref{integrality-up-to-p}. Looking at the interpolation formula for $\mathcal{L}_{p,N}(h)$ one would suspect that the denominator of $\mathcal{L}_{N,p}(h)$ is cleared by an ideal interpolating the values $\langle \g(\varphi),\g(\varphi)\rangle_{ \vert d_K\vert}$, where $\varphi$ is a Hecke character as above. It turns out that this is true  in view of the following result on congruence ideals due to Hida and Tilouine, which allows us to study the denominator of $\cL_{p,N}(h)$. 
\begin{proposition}
\label{integrality-congruence-ideal}
Let $H_\g$ be an annihilator in the congruence ideal associated to the Hida family $\g$ as defined in \cite{HT94}.  For any $h\in S_2(\Gamma_0(N), \cO_v)$, we have
\[H_\g\mathcal{L}_{p,N}(h)\in \Lambda_p\] 

\end{proposition}
\begin{proof}
Our strategy is similar to the one employed in \cite[Sections 4-5]{collins}. Let $\varphi\psi$ be a Hecke character of trivial conductor and infinity type $(-l_2,l_1)$ with $l_2-l_1\ge 1$. Then $\g(\varphi)$ is a newform of level $\vert d_K\vert$ and weight $l_2-l_1+1$. Analogously to \cite[Section 7]{hida88}, we define the $\Lambda_p$-linear map 
\[
\ell_\g:S(\vert d_K\vert, \Lambda_p)\rightarrow \Frac \Lambda_p
\]
whose  specialization at $\varphi\psi$ is given by
\[\ell_{\g(\varphi)}(h(\varphi\psi))=a(1,(e_{\ord} h(\varphi\psi))\big|1_{\g(\varphi)}),\]
where $h(\varphi\psi)$ is a form of level $\vert d_K\vert $ and weight $l_2-l_1+1$, $1_{\g(\varphi)}$ is the projector induced by $\g(\varphi)$ defined in \cite{hida85} and $a(1,f)$ denotes the first Fourier coefficient of $f$. It follows from \cite[Proposition 7.8]{hida88} that $H_{\g}\ell_{\g}$ takes values in $\Lambda_p$. Note that Hida defines the function $\ell_{\g}$ only on the space of ordinary $\Lambda_p$-adic modular forms. But the ordinary projector $e_{\ord}$ induces an injection of Hecke algebras $T^{\ord}\to T$. Thus, we can extend $\ell_\g$ to the space of all modular forms by applying $e_{\ord}$ to  $h(\varphi\psi)$.

Recall that $\g(\varphi)\big|1_{\g(\varphi)}=\g(\varphi)$.  If $g'$ is an eigenform that is not a multiple of $\g(\varphi)$, then $g'\big| 1_{\g(\varphi)}=0$. It follows that 
\[\ell_{\g(\varphi)}(h(\varphi\psi))=\frac{\langle \g(\varphi)^*,e_{\ord}(h(\varphi\psi))\rangle_{\vert d_K\vert}}{\langle \g(\varphi),\g(\varphi)\rangle_{\vert d_K\vert}}.\]

We  extend this linear functional to forms of level $N|d_K|$ as follows. Consider the trace operator
\begin{align*}
    \textup{Tr}\colon S(\Gamma_1(N\vert d_K\vert),\Lambda_p)&\to \ S(\Gamma_1(\vert d_K\vert,\Lambda_p),\\
    f_k&\mapsto \sum_\gamma f\vert_k \gamma,
\end{align*}
 where $f_k$ is the weight $k$ specialization of a Hida family $\mathbf{f}$, $\vert_k \gamma$ denotes the slash action for modular forms and the sum runs over a set of representatives for $\big(\Gamma_1(\vert d_K\vert)\cap \Gamma_0(p)\big)/\big(\Gamma_1(\vert d_K\vert N)\cap \Gamma_0(p)\big)$.
Then 
\begin{align*}
    &\langle \g^*(\varphi),\textup{Tr}(f')\rangle_{\vert d_K\vert}\\
    =\ &\frac{1}{[\Gamma_1(\vert d_K\vert)\cap \Gamma_0(p):\Gamma_1(N\vert d_K\vert) \cap \Gamma_0(p)]}\langle \g^*(\varphi),\textup{Tr}(f')\rangle_{N\vert d_K\vert}\\
    =\ &\langle \g^*(\varphi),f'\rangle_{N\vert d_K\vert}.
\end{align*}

Let $\ell'_{\g}=\ell_{\g}\circ \textup{Tr}$. By definition $H_{\g}\ell'_{\g}$ is integral and its specialization at $\varphi\psi$ is given by
\[\ell'_{\g(\varphi)}(h(\varphi\psi))=\frac{\langle \g(\varphi)^*,e_{\ord}(h(\varphi\psi))\rangle_{N\vert d_K\vert}}{\langle \g(\varphi),\g(\varphi)\rangle_{\vert d_K\vert}},\]
which concludes the proof.
\end{proof}

We review the relation between the congruence ideal $H_\g$ and the anticyclotomic specialization of $L_{K,p}$.

Let $M(\mathcal{K})/\mathcal{K}$ (resp, $M(K^\textup{a}_\infty)/K^\textup{a}_\infty)$) be the maximal abelian pro-$p$ extension unramified outside $\mathfrak{p}$. Then $X(\mathcal{K}):=\Gal(M(\mathcal{K})/\mathcal{K})$ (resp. $X(K^{\textup{a}}_\infty):=\Gal(M(K^\textup{a}_\infty)/K^\textup{a}_\infty)$) is a finitely generated $\Lambda$-torsion (resp. $\Lambda^\textup{a}$-torsion) module.

 Recall from \S\ref{prelim} that $\Delta$ denotes the torsion subgroup of $H_{p^\infty}$. 
 Let $\chi$ be a character of $\Delta$. As $\vert \Delta \vert $ is coprime to $p$, each $\Lambda_p$-module $M$ decomposes into $M=\oplus M_\chi$, where $M_\chi$ is a $\Lambda_p$-module on which $\Delta$ acts via $\chi$. In particular, $\Lambda_\chi$ is the maximal subalgebra of $\Lambda_p$ on which $\Delta$ acts via $\chi$. Let $F_{\chi}$ be a generator of the characteristic ideal of $X(\mathcal{K})_{\chi}$ as a $\Lambda_\chi$-module. 
 
 The following result on the Iwasawa main conjecture over $K$ is due to Rubin.

\begin{theorem}[Rubin]\label{rubin-main-conjecture}
As ideals in $\Lambda_\chi$, we have the equality
\[(F_\chi)=(L_{K,p,\chi}),\]
where $L_{K,p,\chi}$ denotes the $\chi$-isotypic component of $L_{K,p}$.
\end{theorem}
\begin{proof}
This follows from \cite[Theorem 4.1(i)]{rubin91} combined with the remark after Theorem~12.2 of op. cit. 
\end{proof}
Similarly, every $\Lambda^\textup{a}$-module $M$ decomposes as $M=\oplus M_\chi$, where the direct sum runs over all characters of $\Delta$ that factor through $\Gal(K^\textup{a}_\infty/K)$. For any such $\chi$, we denote by $f_\chi$ a generator of  the characteristic ideal of $X(K^\textup{a}_\infty)_\chi$ as a $\Lambda_\chi^\textup{a}$-module. 
The following theorem is due to Hida and Tilouine.

\begin{theorem}\cite[Theorem 0.3]{HT94} \label{theorem:main-conj-congruence-number}Let $\chi$ be a character of $\Delta$ factoring through $\Gal(K^a_\infty/K)$ and let $H_{\g,\chi}$ be the $\chi$-isotypic component of $H_{\g}$. Then as ideals in $\Lambda_\chi^\textup{a}$,
\[H_{\g,\chi} \mid h(K) \mathcal{V}(f_\chi),\]
where $h(K)$ denotes the class number of $K$ and $\mathcal{V}$ is defined as in \S\ref{prelim}.
\end{theorem}

\begin{corollary}
\label{cor:conguence-number}We have
\[\mu (\pi^\textup{a}(H_{\g,\chi}))\le \mu(f_\chi),\]
for all characters of $\Delta$ that factor through $\Gal(K^\textup{a}_\infty/K)$. 
\end{corollary}
\begin{proof}
By Theorem~\ref{theorem:main-conj-congruence-number}, $\pi^{\textup{a}}(H_{\g,\chi})\mid h(K) \pi^{\textup{a}}\mathcal{V}(f_\chi)$ over $\Lambda^\textup{a}_\chi$. As $p$ does not divide $h(K)$ by assumption and $\mu(\pi^{\textup{a}}(\mathcal{V}(G)))=\mu(G)$ for every $G\in \Lambda^{\textup{a}}$, the claim follows.
\end{proof}
Next we prove a relation between $f_\chi$ and $\pi^\textup{a}(F_\chi)$.
\begin{lemma}
\label{mu-invariants-class-groups}Let $\chi$ be a character of $\Delta$ that factors through $\Gal(K^\textup{a}_\infty/K)$. Then \[\mu(\pi^{\textup{a}}(F_\chi))=\mu(f_\chi).\] 
\end{lemma}
\begin{proof}
For every $\Lambda_p$-module $M$ we denote by $\pi\colon M\to M/\ker(\pi)M$ the natural projection induced by $\pi^\textup{a}\colon \Lambda_p\to \Lambda^{\textup{a}}$. Let $M'(\mathcal{K})\subset M(\mathcal{K})$ be the maximal extension that is abelian over $K^\textup{a}_\infty$. Note that $\mathcal{K}\subset M'(\mathcal{K})$ and that $\pi(X(\mathcal{K}))$ is the maximal quotient of $X(\mathcal{K})$ on which $\Gal(\mathcal{K}/K^\textup{a}_\infty)$ acts trivially. It follows that $\pi(X(\mathcal{K}))=\Gal(M'(\mathcal{K})/\mathcal{K})$. As $\mathcal{K}/K^\textup{a}_\infty$ is unramified we see that $M'(\mathcal{K})=M(K^\textup{a}_\infty)$. Thus, we have an exact sequence. 
\[0\to \pi(X(\mathcal{K}))_\chi\to X(K^\textup{a}_\infty)_\chi\to \Gal(K_\infty/K^\textup{a}_\infty)^{\textup{pro-}p}.\]
The lemma now follows from the fact that $\Gal(K_\infty/K^\textup{a}_\infty)^{\textup{pro-}p}$ is finitely generated over $\Zp$ (so its $\mu$-invariant vanishes).
\end{proof}
\begin{corollary} 
\label{cor:mu-invariant-congruence-ideal-and-katz}Let $\mathbbm{1}$ be the trivial character of $\Delta$. Then we have 
\[\mu(\pi^\ac(H_{\g,\mathbbm{1}}))\le\mu(L^\textup{ac}_{K,p}).\]
\end{corollary}
\begin{proof}
It follows from Theorem~\ref{rubin-main-conjecture}, Corollary~\ref{cor:conguence-number} and Lemma~\ref{mu-invariants-class-groups} that
\[\mu (\pi^\ac(H_{\g,\mathbbm{1}}))\le \mu(\pi^\ac(L_{K,p,\mathbbm{1}})).\]
By the definition of $L^\ac_{p,K}$, we have $\mu(\pi^\ac(L_{K,p,\mathbbm{1}}))\le \mu(L^\ac_{p,K})$, which concludes the proof.
\end{proof}

We can now deduce the following integral version of Corollary~\ref{integrality-up-to-p}.
\begin{corollary}\label{cor:integrality} Let $h$ be as in Corollary \ref{integrality-up-to-p}. Then
\[\mathcal{L}^\textup{ac}_{p,N}(h)L^\textup{ac}_{K,p}\in \Lambda^{ac}.\]
\end{corollary}
\begin{proof}
     This follows from Proposition \ref{integrality-congruence-ideal} combined with Corollaries \ref{integrality-up-to-p} and \ref{cor:mu-invariant-congruence-ideal-and-katz}.
\end{proof}

\subsection{Proof of Theorem~\ref{thmB}}
We are now ready to give the proof of Theorem~\ref{thmB} stated in the introduction. 
\begin{theorem}
\label{main-theorem-analytic}
Let $N_1$ and $N_2$ be two positive integers that are coprime to $p\cdot d_K$ satisfying \ref{GHH}. Assume that $N_i$ satisfies \ref{sq-free}. For $\is$, let $f_i=\sum b^{(i)}_qq^n\in S_2(\Gamma_0(N_i),\cO_v)$ be a newform satisfying  \eqref{cong}, \ref{div}, and \eqref{unit}, with respect to $N'=N_i$ and $N$ chosen to be a common multiple of $N_1$ and $N_2$ that is also a square.  
 
Then  $\alpha(f_1)$ is a $p$-adic unit and $\mu(L_p^\BDP(f_1))=0$ if and only if $\alpha(f_2)$ is a $p$-adic unit and $\mu(L_p^\BDP(f_2))=0$. Furthermore, when this holds, \[2\lambda(L_p^\BDP(f_1))+\sum_{\mathfrak{q}\mid\tilde{N}^{(1)}_3}\lambda(\mathcal{P}^{f_1}_\mathfrak{q})=2\lambda(L_p^\BDP(f_2))+\sum_{\mathfrak{q}\mid\tilde{N}^{(2)}_3}\lambda(\mathcal{P}^{f_2}_\mathfrak{q}).\]
\end{theorem}

\begin{proof}
The hypothesis \eqref{cong} implies that $b_q^{(1)}\equiv b_q^{(2)} \pmod {\varpi}$ for all $q$ that are coprime to $pN$. If $f_i$ are non-ordinary at $p$, then $b^{(1)}_p\equiv b^{(2)}_p \equiv 0\ \pmod {\varpi}$. If $f_i$ are ordinary at $p$,  the congruence $b_p^{(1)}\equiv b_p^{(2)}\pmod{\varpi}$ follows from \cite[Theorem 2]{wiles88}. In other words, the Fourier coefficients of $f_1$ and $f_2$ are congruent modulo $\varpi$ at primes outside $N$.
Let $\check{f}_1$ and $\check{f}_2$  be the elements of $S_2(\Gamma_0(N),\cO_v)$ given by Definition \ref{def:fcheck}. All their Fourier coefficients are congruent modulo $\varpi$. In particular, there exists $h\in S_2(\Gamma_0(N),\cO_v)$ such that $\check{f}_1-\check{f}_2=\varpi h$.

By Corollary \ref{cor:integrality} we see that $L^\textup{ac}_{K,p}\mathcal{L}^\textup{ac}(h)\in \Lambda^\textup{ac}$. As discussed in Remark~\ref{rem:linearity}, we have
\begin{equation}
    \label{congruence-of-l-functions}
\left(\mathcal{L}^\textup{ac}_{p,N}(\check{f}_1)-\mathcal{L}^\textup{ac}_{p,N}(\check{f}_2)\right)L^\textup{ac}_{K,p}=\varpi L^\textup{ac}_{K,p}\mathcal{L}^\textup{ac}_{p,N}(h)\in \varpi\Lambda^\textup{ac}.\end{equation}
Thus, 
\[
\mu(\mathcal{L}^\textup{ac}_{p,N}(\check{f}_1)L^\textup{ac}_{K,p})=0\Leftrightarrow \mu(\mathcal{L}^\textup{ac}_{p,N}(\check{f}_2)L^\textup{ac}_{K,p})=0.
\]
Consequently, Theorem~\ref{BDP-l-function-for-eigenforms}(b) implies that
\[
\mu(\mathcal{L}^\textup{ac}_{p,N_1}(\check{f}_1)L^\textup{ac}_{K,p})=0\Leftrightarrow \mu(\mathcal{L}^\textup{ac}_{p,N_2}(\check{f}_2)L^\textup{ac}_{K,p})=0.
\]
Therefore, the first assertion of the theorem follows from 
Theorems \ref{thm:integrality-initial-case}.

When these equivalent conditions hold, equation \eqref{congruence-of-l-functions} implies that
\[\lambda(\mathcal{L}^\textup{ac}_{p,N}(\check{f}_1)L^\textup{ac}_{K,p})=\lambda(\mathcal{L}^\textup{ac}_{p,N}(\check{f}_2)L^\textup{ac}_{K,p}).\]
Therefore, combining Theorems \ref{thm:integrality-initial-case} and \ref{BDP-l-function-for-eigenforms}(b)  gives rise to the equations
\begin{align*}
    2\lambda({L}_p^\BDP(f_1))+\sum_{q\mid\tilde{N}^{(1)}_3}\sum_{\mathfrak{q}\mid q}\lambda(\mathcal{P}^{f_1}_\mathfrak{q})=\ &\lambda(\mathcal{L}^\ac_{p,N_1}(f_1)L^\textup{ac}_{K,p})+\sum_{q\mid\tilde{N}^{(1)}_3}\sum_{\mathfrak{q}\mid q}\lambda(\mathcal{P}^{f_1}_\mathfrak{q})\\=\ &\lambda(\mathcal{L}^\textup{ac}_{p,N}(\check{f}_1)L^\textup{ac}_{K,p})\\
    =\ &\lambda(\mathcal{L}^\textup{ac}_{p,N}(\check{f}_2)L^\textup{ac}_{K,p})\\
    =\ &\lambda(\mathcal{L}_{p,N_2}^\ac(f_2)L^\textup{ac}_{K,p})+\sum_{q\mid\tilde{N}^{(2)}_3}\sum_{\mathfrak{q}\mid q}\lambda(\mathcal{P}^{f_2}_\mathfrak{q})\\
    =\ &2\lambda(L_p^\BDP(f_2))+\sum_{q\mid\tilde{N}^{(2)}_3}\sum_{\mathfrak{q}\mid q}\lambda(\mathcal{P}^{f_2}_\mathfrak{q}),
\end{align*} which concludes the proof.
\end{proof}
\begin{remark}
\label{rem:relaxed-condition}
For $1\le j\le 3$ let $M^{(i)}_j\mid \tilde{N}_j^{(i)}$ be the product of all primes dividing $\tilde{N}^{(i)}_3$ such that $b^{(1)}_q\not\equiv b^{(2)}_q\pmod\varpi$ and redefine $N$ as the greatest common multiple of $N_1$ and $N_2$ such that $q^2\mid N$ for each prime $q$ that divides one of the numbers $M^{(i)}_j$ for $i\in \{1,2\}$ and $1\le j\le 3$. These choices are enough to ensure that $\check{f_1}$ and $\check{f_2}$ have congruent Fourier-coefficients everywhere. 
\end{remark}
Analogous to Lemma~\ref{lem:example}, we deduce the following application of Theorem~\ref{main-theorem-analytic}.

\begin{lemma}\label{lem:example2}
    Let $f$ be a newform of level $N$ and weight $2$ and let $\chi$ be a quadratic character such that $\chi(p)=-1$. Assume that $p\in \{3,5\}$ and that $f$ has partial Eisenstein decent by $(\chi,\chi,N_1,N_2,N_0)$. 
    Let $K$ be an imaginary quadratic field so that  \ref{GHH} holds with respect to $N\textup{cond}(\chi)$. We assume that $N^-=1$ and that $11$ and $19$ split in $K$.  Then 
    \[\mu(L_p^\BDP(f))=0.\]
\end{lemma}
\begin{proof}
    Let $E$ be the elliptic curve and $E^\chi$ its quadratic twist, given as in the proof of Lemma \ref{lem:example}. Let $g$ be the weight-two modular form corresponding to $E^{\chi}$. By \cite[Theorem 2.2.2]{CGLS} $\mu(L^\BDP_p(g))=0$. As $N^-=1$ we know that $\alpha(g)=1$ and that \ref{div} and \eqref{unit} are satisfied.   Thus, we can apply Theorem \ref{main-theorem-analytic} to $f$ and $g$.
 \end{proof}

\subsection{An application: congruence relations between Heegner points}
For $i \in \{1, 2\}$, let $f_i \in S_2 (\Gamma_0 (N_i), \cO_v)$ be a newform. Recall from Theorem~\ref{thm:integrality-initial-case} that up to an explicit unit $u$ of $\Lambda^\ac$ only depending on $N$ and the imaginary quadratic field $K$, we have
\[\mathcal{L}_{p,N}^\textup{ac}(f_i)L^\textup{ac}_{K,p}
=
u  \alpha(f_i) L_p^\BDP(f_i)^2.\]
We choose  $N$ to be a common multiple of $N_1$ and $N_2$ that is a perfect square. Let $\check{f_i} \in S_2 (\Gamma_0 (N), \cO_v)$ be the oldform of level $N$ corresponding to the newform $f_i$ as in Definition~\ref{def:fcheck} .

Recall from the proof of Theorem~\ref{BDP-l-function-for-eigenforms} that for an anticyclotomic Hecke character $\varphi\psi$  of trivial conductor and infinity type $(b,-b)$, there is an explicit integer $c(f_i) \in \cO_v$ independent of $\varphi\psi$ given by a product of Euler factors at primes that are inert in $K/\Q$ such that the equation
\[C({\g}(\varphi),\check{f_i},N,b + 1) = c(f_i) \prod_{\mathfrak{q}\mid M_3^{(i)}}\mathcal{P}^{f_i}_\mathfrak{q}C({\g}(\varphi),f_i,N_i, b + 1)\]
holds. The proof of Theorem~\ref{main-theorem-analytic} implies the following congruence result for BDP $p$-adic $L$-functions.

\begin{proposition} 
\label{prop:congruence-l-functions}
    Let $f_1$ and $f_2$ be two newforms satisfying the assumptions in Remark \ref{rem:relaxed-condition} and all other conditions of Theorem~\ref{main-theorem-analytic}. Then we have the congruence relation 
    \begin{align*} &\alpha(f_1)c (f_1) \prod_{\mathfrak{q}\mid M^{(1)}_3}\mathcal{P}^{f_1}_\mathfrak{q} 
    \sigma^{-1}_{\mathcal{N}_1^+}L_p^\BDP(f_1)^2 (\varphi\psi)\\
        &\equiv \sigma^{-1}_{\mathcal{N}_2^+}\alpha(f_2) c (f_2)\prod_{\mathfrak{q}\mid \tilde{M}^{(2)}_3}\mathcal{P}^{f_2}_\mathfrak{q}
    L_p^\BDP(f_2)^2 (\varphi\psi)
    \pmod{\varpi}.\end{align*}
\end{proposition}

Recall that $f_i$ was chosen to be a new form of level $N_i=N^+_iN^-_i$. Let $B_i$ be the indefinite quaternion algebra of conductor $N^-_i$ over $\Q$ and let $\textup{Sh}_{B_i}$ be the Shimura curve of level $\Gamma_1(N_i^+)$. We denote its Jacobian by $J_{B_i}$. Let $f_i^\textup{JL}$ be the Jacquet-Langlands transfer of $f_i$ to $\textup{Sh}_{B_i}$, normalized by requiring the corresponding section of a certain canonical line bundle $\cL$ over a canonical proper smooth model of $\textup{Sh}_{B_i}$ over $\bZ_p$ to be integral and primitive with respect to the integral structure of $\cL$ (see \cite[Section 2]{prasanna2006} for details).
The modular form $f_i^\textup{JL}$ induces a differential form $\omega_{f_i}$ on $\textup{Sh}_{B_i}$ and $J_{B_i}$. 

We fix an embedding $K\hookrightarrow B_i$. Let $\tau$ be its fixed point in the upper half plane and we let $(A,t,\omega)$ be the triple associated to $\tau$, where $A$ denotes a false elliptic curve, $t$ a level $N^+$-structure and $\omega$ a global non-vanishing differential (see \cite[Section 2]{HB15} for details). There is a well-defined action of the ideal class group $\textup{Cl}_K$ on the equivalence classes of triples  $[(A,t,\omega)]$.
For each ideal class $[\mathfrak{A}]$ coprime to $N^+$ we define the CM point $P_\mathfrak{A}$ on $\textup{Sh}_{B_i}$ to be the point corresponing to $[\mathfrak{A}(A,t,\omega)]$. We define the Heegner point
$\Delta_{f_i}$ associated to $f_i$ as the sum
\[\sum_{[\mathfrak{A}]\in \textup{Cl}_K} (P_\mathfrak{A}-\infty) ,\]
which is a priori defined over a finite extension of $K$ containing the Fourier-coefficients of $f_i$ and a quadratic extension only depending on $N_i$. As we are only working with weight-two forms, one can actually show that it is defined over $K$ (see  \cite[discussion   after Proposition 8.12]{HB15}).
 Translating \cite[Proposition~ 8.12]{HB15} into our setting, we obtain the following relation.
\begin{theorem}
\label{relation-heegner-bdp-function}
    Let $\mathbbm{1}$ be the trivial character on $\Gamma^\ac$. Then 
    \begin{align*}(L_p^{\BDP}(f)(\mathbbm{1}))^2=(1 - a_p p^{-1}+  p^{- 1})^2 \log_{\omega_{f_i}}(\Delta_{f_i})^2.
        \end{align*} 
\end{theorem}

We conclude this section with the following partial generalization of \cite[Theorem~3.9]{KL}.
\begin{theorem}\label{thm:KL}
   Let $f_1$ and $f_2$ be newforms of level $N_1$ and $N_2$ respectively. Assume that the conditions of Remark \ref{rem:relaxed-condition} and all the conditions of Theorem~\ref{main-theorem-analytic} are satisfied. Then 
   \begin{align*}&c(f_1)\alpha(f_1)\prod_{q\mid {M}^{(1)}_3}\prod_{\mathfrak{q}\mid q}\mathcal{P}^{f_1}_\mathfrak{q}(0)(1 - b^{(1)}_p p^{-1}+  p^{- 1})^2\log_{\omega_{f_1}}(\Delta_{f_1})^2\\&\equiv c(f_2)\alpha(f_2)\prod_{q\mid {M}^{(2)}_3}\prod_{\mathfrak{q}\mid q}\mathcal{P}^{f_2}_\mathfrak{q}(0)(1 - b^{(2)}_p p^{-1}+  p^{- 1})^2\log_{\omega_{f_2}}(\Delta_{f_2})^2 \pmod \varpi.
   \end{align*}
\end{theorem}
\begin{proof}
    This is an immediate consequence of Theorems \ref{prop:congruence-l-functions} and \ref{relation-heegner-bdp-function}.   
\end{proof}
\begin{remark}Note that for $q$ such that $q\mid {M}^{(i)}_3$ and $q\not \mid N_i$ we have
    \[\prod_{\mathfrak{q}\mid q}\mathcal{P}^{f_i}_\mathfrak{q}(0)=(1-b^{(i)}_qq^{-1}+q^{-1})^2.\] If $q\mid {M}^{(i)}_3$ and $q\mid N_i$,  then we have
    \[\prod_{\mathfrak{q}\mid q}\mathcal{P}^{f_i}_\mathfrak{q}(0)=(1-b^{(i)}_qq^{-1})^2.\]
    Thus, if $N^-=1$ we recover the congruence relation proved in \cite[Theorem 3.9]{KL} for $m=1$ at the trivial character.

    Unfortunately, this generalized congruence relation does not allow us to improve their lower bound on the Goldfeld conjecture: Let $E$ be an elliptic curve and let $d$ be a square free integer that is coprime to the conductor $N$ of $E$. Then the conductor of the twisted elliptic curve $E^d$ is $d^2 N$. Let $K$ be an imaginary quadratic field simultaneously satisfying \ref{GHH} with respect to $N$ and $d^2 N$. Then all prime factors of $d$ have to split in $K$. Thus, we obtain the same lower bound as in \cite[Theorem 1.12]{KL}.
   
\end{remark}

\section{The BDP main conjecture and congruent modular forms}

We now prove Theorem~\ref{thmC} stated in the introduction.
\begin{theorem}\label{main}
 Let $f_1$ and $f_2$ be two newforms of level $N_1$ and $N_2$, respectively, satisfying \ref{GHH}, \eqref{cong}, \ref{H0}, \ref{div}, \ref{sq-free} and \eqref{unit}. Assume that $\mu(L_p^\BDP(f_1))=0$ and that $\alpha(f_1)$ is a $p$-adic unit. If one inclusion in \eqref{IMC} holds for $f_2$ and that the full equality of \eqref{IMC}  holds for $f_1$, then the full equality of \eqref{IMC} holds for $f_2$. 
\end{theorem}
\begin{proof}
The assumption that \eqref{IMC} holds for $f_1$ implies that $\Sel^\BDP(K_\infty,A_1)$ is $\Lambda^\ac_{\cO_v}$-cotorsion and
\[2\lambda(L_p^\BDP(f_1))=\lambda(\Sel^\BDP(K_\infty,A_1)^\vee).\]
Furthermore,  we have
\[
\mu(L_p^\BDP(f_1))=\mu(\Sel^\BDP(K_\infty,A_1)^\vee)=0.
\]

 On the one hand, Corollary \ref{cor-equality-of-invariants} tells us that
\[\mu(\Sel^\BDP(K_\infty,A_2)^\vee)=0\]
and
\begin{equation}
\label{eq:iwasawa-invariants-algebriac}
    \lambda(\Sel^{\BDP}(K_\infty,A_1)^\vee)+\sum_{l|\Sigma_0}\lambda(\mathcal{P}^{f_1}_{l})=\lambda(\Sel^{\BDP}(K_\infty,A_2)^\vee)+\sum_{l| \Sigma_0}\lambda(\mathcal{P}^{f_2}_{l}).
\end{equation}
Theorem \ref{main-theorem-analytic} tells us that $\mu(L_p^\BDP(f_2))=0$ and that 
\begin{equation}
    \label{eq:iwaswa-invariants-analytic}
    2\lambda(L_p^\BDP(f_1))+\sum_{\mathfrak{q}\mid\tilde{N}^{(1)}_3}\lambda(\mathcal{P}^{f_1}_\mathfrak{q})=2\lambda(L_p^\BDP(f_2))+\sum_{\mathfrak{q}\mid\tilde{N}^{(2)}_3}\lambda(\mathcal{P}^{f_2}_\mathfrak{q}).
\end{equation}

As we assume that one inclusion  of \eqref{IMC} holds for $f_2$ and $\mu(L_p^\BDP(f_2))=\mu(\Sel^\BDP(K_\infty,A_2)^\vee)=0$, to show that the full equality of \eqref{IMC} holds,  it is enough to show that $2\lambda(L_p^\BDP(f_2))$ equals $\lambda(\Sel^\BDP(K_\infty,A_2)^\vee)$.

Recall that for $\is$, $\mathcal{P}_\mathfrak{q}^{f_i}=1$ for all $v\in \Sigma_0$ such that $b^{(i)}_q=0$. Thus, $\sum_{\mathfrak{q}\mid\tilde{N}^{(i)}_3}\lambda(\mathcal{P}^{f_i}_\mathfrak{q})=\sum_{l| \Sigma_0}\lambda(\mathcal{P}^{f_i}_{l})$.
Combining equations \eqref{eq:iwasawa-invariants-algebriac} and \eqref{eq:iwaswa-invariants-analytic} gives
\begin{align*}
    \lambda(\Sel^{\BDP}(K_\infty,A_2)^\vee)+\sum_{l| \Sigma_0}\lambda(\mathcal{P}^{f_2}_{l})=\ & \lambda(\Sel^{\BDP}(K_\infty,A_1)^\vee)+\sum_{l| \Sigma_0}\lambda(\mathcal{P}^{f_1}_{l})\\
    =\ & \lambda(\Sel^{\BDP}(K_\infty,A_1)^\vee)+\sum_{\mathfrak{q}\mid\tilde{N}^{(1)}_3}\lambda(\mathcal{P}^{f_1}_\mathfrak{q})\\
    =\ & 2\lambda(L_p^\BDP(f_1))+\sum_{\mathfrak{q}\mid\tilde{N}^{(1)}_3}\lambda(\mathcal{P}^{f_1}_\mathfrak{q})\\
    =\ & 2\lambda(L_p^\BDP(f_2))+\sum_{\mathfrak{q}\mid\tilde{N}^{(2)}_3}\lambda(\mathcal{P}^{f_2}_\mathfrak{q})\\
    =\ & 2\lambda(L_p^\BDP(f_2))+\sum_{l|\Sigma_0}\lambda(\mathcal{P}^{f_2}_{l}).
\end{align*}
Thus,
\[\lambda(\Sel^{\BDP}(K_\infty,A_2)^\vee)=2\lambda(L_p^\BDP(f_2)),\] which concludes the proof. 
\end{proof}
\begin{remark}
Let $f_1$ and $f_2$ be newforms of level $N_1$ and $N_2$, respectively, satisfying \ref{GHH}, \eqref{cong}, \ref{H0}, \ref{div}, \ref{sq-free} and \eqref{unit}. Assume that the $\mu(L_p^\BDP(f_1))=0$ and that $\alpha(f_1)$ is a $p$-adic unit. Assume that the full equality of \eqref{IMC}  holds for $f_1$. Below we give a list of (additional) sufficient conditions to make sure that one inclusion of \eqref{IMC} holds for $f_2$. In particular, \eqref{IMC} holds for $f_2$ in these cases. 
\begin{itemize}
    \item Assume that $p$ is an ordinary prime for $f_2$ and that its residual representation $\overline{\rho_{2}}$ is absolutely irreducible. By  \cite[Theorem 5.2]{BCK21}, one implication of \eqref{IMC} is true if and only if one implication of the Perrin-Riou Heegner point main conjecture is true. According to \cite[Theorem E]{naomi} there are non-negative integers $r$ and $s$ such that 
    \[\Lambda^\ac\varpi^{r}\textup{Char}_{\Lambda^{ac}_{\cO_v}}(\Sel^\BDP(K_\infty,A_{f_2})^\vee)=\Lambda^{\ac}\varpi^{s}(L_p^\BDP(f_2))^2.\]
    Under the assumptions of Theorem \ref{main} all $\mu$-invariants vanish and we obtain $r=s=0$.
    \item Assume that $f_2$ is non-ordinary at $p$ and that $N_2^-=1$. Then \cite[Theorem~1.5]{kobayashiota}  implies
\[
{L}_p^\BDP(f_2)^2\in \Lambda^\ac\varpi^l\Char_{\Lambda_{\cO_v}^\ac}\Sel^\BDP(K_\infty,A_2)^\vee,
\]
 for some non-negative integer $l$. Using again that all $\mu$-invariants vanish we can assume that $l=0$. Note that under additional assumptions on $f_2$ \cite[Theorem A]{lei-zhao} proves one inclusion of \eqref{IMC} without the additional factor $\varpi^l$.
\item Assume that $f_2$ is non-ordinary at $p\ge 5$ and that $f_2$ satisfies \begin{itemize}
    \item $N_2^-\neq 1$
    \item If $N_2$ is odd, $2$ splits in $K$
    \item $N_2$ is square free.
    \item $f_2$ is associated to an elliptic curve $E/\Q$.
\end{itemize} Then  \cite[Theorem 6.8 and Theorem C]{castellawan1607} imply that there is a non-negative integer $r$ such that 
\[\varpi^r\Lambda^\ac\Char_{\Lambda_{\cO_v}^\ac}\Sel^\BDP(K_\infty,A_2)^\vee\in {L}^\BDP_p(f_2)^2\Lambda^\ac.\]
As we have already seen in the proof of Theorem~\ref{main} that in our setting all $\mu$-invariants vanish, we can choose $r=0$ and obtain that one inclusion of \eqref{IMC} holds for $f_2$.
\end{itemize}
\end{remark}

\bibliographystyle{amsalpha}
\bibliography{references}

\providecommand{\bysame}{\leavevmode\hbox to3em{\hrulefill}\thinspace}
\providecommand{\MR}{\relax\ifhmode\unskip\space\fi MR }
\providecommand{\MRhref}[2]{%
  \href{http://www.ams.org/mathscinet-getitem?mr=#1}{#2}
}
\providecommand{\href}[2]{#2}
\begin{thebibliography}{CGLS22}

\bibitem[BCK21]{BCK21}
Ashay Burungale, Francesc Castella, and Chan-Ho Kim, \emph{A proof of
  {P}errin-{R}iou's {H}eegner point main conjecture}, Algebra Number Theory
  \textbf{15} (2021), no.~7, 1627--1653.

\bibitem[BD05]{BertoliniDarmon2005}
M.~Bertolini and H.~Darmon, \emph{Iwasawa's main conjecture for elliptic curves
  over anticyclotomic {$\Bbb Z_p$}-extensions}, Ann. of Math. (2) \textbf{162}
  (2005), no.~1, 1--64.

\bibitem[BDP13]{bertolinidarmonprasanna13}
Massimo {Bertolini}, Henri {Darmon}, and Kartik {Prasanna}, \emph{{Generalized
  Heegner cycles and $p$-adic Rankin $L$-series.}}, {Duke Math. J.}
  \textbf{162} (2013), no.~6, 1033--1148.

\bibitem[Bra11]{miljan}
Miljan Brako\v{c}evi\'{c}, \emph{Anticyclotomic {$p$}-adic {$L$}-function of
  central critical {R}ankin-{S}elberg {$L$}-value}, Int. Math. Res. Not. IMRN
  (2011), no.~21, 4967--5018.

\bibitem[Bri07]{brink}
David Brink, \emph{Prime decomposition in the anti-cyclotomic extension}, Math.
  Comp. \textbf{76} (2007), no.~260, 2127--2138.

\bibitem[Bur17]{burungale2}
Ashay Burungale, \emph{On the non-triviality of the {$p$}-adic {A}bel-{J}acobi
  image of generalised {H}eegner cycles modulo {$p$}, {II}: {S}himura curves},
  J. Inst. Math. Jussieu \textbf{16} (2017), no.~1, 189--222.

\bibitem[Cas13]{castella13}
Francesc Castella, \emph{Heegner cycles and higher weight specializations of
  big {H}eegner points}, Math. Ann. \textbf{356} (2013), no.~4, 1247--1282.

\bibitem[C{\c{C}}SS18]{CCSS}
Francesc Castella, Mirela {\c{C}}iperiani, Christopher Skinner, and Florian
  Sprung, \emph{On the {I}wasawa main conjectures for modular forms at
  non-ordinary primes}, 2018, preprint, arXiv:1804.10993.

\bibitem[CGLS22]{CGLS}
Francesc Castella, Giada Grossi, Jaehoon Lee, and Christopher Skinner, \emph{On
  the anticyclotomic {I}wasawa theory of rational elliptic curves at
  {E}isenstein primes}, Invent. Math. \textbf{227} (2022), no.~2, 517--580.

\bibitem[CKL17]{CKL}
Francesc Castella, Chan-Ho Kim, and Matteo Longo, \emph{Variation of
  anticyclotomic {I}wasawa invariants in {H}ida families}, Algebra Number
  Theory \textbf{11} (2017), no.~10, 2339--2368.

\bibitem[Col20]{collins}
Dan~J. Collins, \emph{Anticyclotomic {$p$}-adic {$L$}-functions and {I}chino's
  formula}, Ann. Math. Qu\'{e}. \textbf{44} (2020), no.~1, 27--89.

\bibitem[CR18]{CR}
Daniele Casazza and Victor Rotger, \emph{Stark points and the {H}ida-{R}ankin
  {$p$}-adic {$L$}-function}, Ramanujan J. \textbf{45} (2018), no.~2, 451--473.

\bibitem[CW22a]{castellawan}
Francesc Castella and Xin Wan, \emph{The {I}wasawa main conjectures for {$\rm
  GL_2$} and derivatives of {$p$}-adic {$L$}-functions}, Adv. Math.
  \textbf{400} (2022), Paper No. 108266, 45.

\bibitem[CW22b]{castellawan1607}
\bysame, \emph{{P}errin-{R}iou's main conjecture for elliptic curves at
  supersingular primes.}, preprint,
  \url{https://web.math.ucsb.edu/~castella/Perrin-Riou.pdf}, 2022.

\bibitem[Del69]{deligne69}
Pierre Deligne, \emph{Formes modulaires et repr\'esentations $\ell$-adiques},
  S\'eminaire Bourbaki (1968/69), no.~21, Exp.\ No.\ 355, 139--172.

\bibitem[DT94]{Diamond-Taylor}
Fred Diamond and Richard Taylor, \emph{Nonoptimal levels of mod {$l$} modular
  representations}, Invent. Math. \textbf{115} (1994), no.~3, 435--462.
  \MR{1262939}

\bibitem[Eme02]{emerton02}
Matthew Emerton, \emph{Supersingular elliptic curves, theta series and weight
  two modular forms}, J. Amer. Math. Soc. \textbf{15} (2002), no.~3, 671--714.

\bibitem[Gre89]{greenberg89}
Ralph Greenberg, \emph{Iwasawa theory for {$p$}-adic representations},
  Algebraic number theory, Adv. Stud. Pure Math., vol.~17, Academic Press,
  Boston, MA, 1989, pp.~97--137.

\bibitem[Gre99]{Greenberg}
\bysame, \emph{Iwasawa theory for elliptic curves}, Arithmetic theory of
  elliptic curves ({C}etraro, 1997), Lecture Notes in Math., vol. 1716,
  Springer, Berlin, 1999, pp.~51--144.

\bibitem[GV00]{greenberg-vatsal}
Ralph Greenberg and Vinayak Vatsal, \emph{On the {I}wasawa invariants of
  elliptic curves}, Invent. Math. \textbf{142} (2000), no.~1, 17--63.

\bibitem[HB15]{HB15}
Ernest Hunter~Brooks, \emph{Shimura curves and special values of {$p$}-adic
  {$L$}-functions}, Int. Math. Res. Not. IMRN (2015), no.~12, 4177--4241.

\bibitem[Hid85]{hida85}
Haruzo Hida, \emph{A {$p$}-adic measure attached to the zeta functions
  associated with two elliptic modular forms. {I}}, Invent. Math. \textbf{79}
  (1985), no.~1, 159--195.

\bibitem[Hid88]{hida88}
\bysame, \emph{A {$p$}-adic measure attached to the zeta functions associated
  with two elliptic modular forms. {II}}, Ann. Inst. Fourier (Grenoble)
  \textbf{38} (1988), no.~3, 1--83.

\bibitem[HL19a]{hatleylei2019}
Jeffrey Hatley and Antonio Lei, \emph{Arithmetic properties of signed {Selmer}
  groups at non-ordinary primes}, no.~3, 1259--1294.

\bibitem[HL19b]{HL2}
\bysame, \emph{Comparing anticyclotomic {S}elmer groups of positive coranks for
  congruent modular forms}, Math. Res. Lett. \textbf{26} (2019), no.~4,
  1115--1144.

\bibitem[HL22]{HL21}
\bysame, \emph{The vanishing of anticyclotomic $\mu$-invariants for
  non-ordinary modular forms}, 2022, to appear in C. R. Math. Acad. Sci. Paris,
  available on arXiv:2108.05958.

\bibitem[How04a]{howardcompositio1}
Benjamin Howard, \emph{The {H}eegner point {K}olyvagin system}, Compos. Math.
  \textbf{140} (2004), no.~6, 1439--1472.

\bibitem[How04b]{howard04}
\bysame, \emph{Iwasawa theory of {H}eegner points on abelian varieties of {$\rm
  GL_2$} type}, Duke Math. J. \textbf{124} (2004), no.~1, 1--45. \MR{2072210}

\bibitem[{Hsi}14]{hsiehnonvanishing}
Ming-Lun {Hsieh}, \emph{{Special values of anticyclotomic Rankin-Selberg
  $L$-functions}}, {Doc. Math.} \textbf{19} (2014), 709--767.

\bibitem[HT94]{HT94}
H.~Hida and J.~Tilouine, \emph{On the anticyclotomic main conjecture for {CM}
  fields}, Invent. Math. \textbf{117} (1994), no.~1, 89--147.

\bibitem[Kat78]{katz78}
Nicholas~M. Katz, \emph{{$p$}-adic {$L$}-functions for {CM} fields}, Invent.
  Math. \textbf{49} (1978), no.~3, 199--297.

\bibitem[Kid18]{kidwell}
Keenan Kidwell, \emph{On the structure of {S}elmer groups of {$p$}-ordinary
  modular forms over {${\bf{Z}}_p$}-extensions}, J. Number Theory \textbf{187}
  (2018), 296--331.

\bibitem[Kim09]{kim09}
Byoung~Du Kim, \emph{The {I}wasawa invariants of the plus/minus {S}elmer
  groups}, Asian J. Math. \textbf{13} (2009), no.~2, 181--190.

\bibitem[Kim17]{kim17}
Chan-Ho Kim, \emph{Anticyclotomic {I}wasawa invariants and congruences of
  modular forms}, Asian J. Math. \textbf{21} (2017), no.~3, 499--530.

\bibitem[KL19]{KL}
Daniel Kriz and Chao Li, \emph{Goldfeld's conjecture and congruences between
  {H}eegner points}, Forum Math. Sigma \textbf{7} (2019), Paper No. e15, 80.

\bibitem[KO20]{kobayashiota}
Shinichi Kobayashi and Kazuto Ota, \emph{Anticyclotomic main conjecture for
  modular forms and integral {P}errin-{R}iou twists}, In Development of Iwasawa
  Theory – the Centennial of K. Iwasawa's Birth (M. Kurihara et al, eds),
  Adv. Stud. Pure Math., Math. Soc. Japan (2020), 537--594.

\bibitem[Kri16]{kriz}
Daniel Kriz, \emph{Generalized {H}eegner cycles at {E}isenstein primes and the
  {K}atz {$p$}-adic {$L$}-{\it function}}, Algebra Number Theory \textbf{10}
  (2016), no.~2, 309--374.

\bibitem[LLZ14]{LLZ1}
Antonio Lei, David Loeffler, and Sarah~Livia Zerbes, \emph{Euler systems for
  {R}ankin-{S}elberg convolutions of modular forms}, Ann. of Math. (2)
  \textbf{180} (2014), no.~2, 653--771.

\bibitem[LLZ15]{LLZ2}
\bysame, \emph{Euler systems for modular forms over imaginary quadratic
  fields}, Compos. Math. \textbf{151} (2015), no.~9, 1585--1625.

\bibitem[LZ22]{lei-zhao}
Antonio Lei and Luochen Zhao, \emph{{On the BDP Iwasawa main conjecture for
  modular forms}}, preprint, arXiv:2211.04377, $\geq$ 2022.

\bibitem[Mar17]{martin17}
Kimball Martin, \emph{The {J}acquet-{L}anglands correspondence, {E}isenstein
  congruences, and integral {$L$}-values in weight 2}, Math. Res. Lett.
  \textbf{24} (2017), no.~6, 1775--1795.

\bibitem[PR87]{perrinriou87}
Bernadette Perrin-Riou, \emph{Fonctions {$L$} {$p$}-adiques, th\'{e}orie
  d'{I}wasawa et points de {H}eegner}, Bull. Soc. Math. France \textbf{115}
  (1987), no.~4, 399--456.

\bibitem[Pra06]{prasanna2006}
Kartik Prasanna, \emph{Integrality of a ratio of {P}etersson norms and
  level-lowering congruences}, Ann. of Math. (2) \textbf{163} (2006), no.~3,
  901--967.

\bibitem[PW11]{pollack_weston_2011}
Robert Pollack and Tom Weston, \emph{On anticyclotomic $\mu$-invariants of
  modular forms}, Compositio Mathematica \textbf{147} (2011), no.~5,
  1353–1381.

\bibitem[Rub91]{rubin91}
Karl Rubin, \emph{The ``main conjecture" of {I}wasawa theory for imaginary
  quadratic fields}, Invent. Math. \textbf{103} (1991), no.~1, 25--68.

\bibitem[Swe22]{naomi}
Naomi Sweeting, \emph{{Kolyvagin's conjecture, bipartite Euler systems, and
  heigher congruences of modular forms}}, preprint, arXiv:2012.11771, $\geq$
  2022.

\bibitem[Wil88]{wiles88}
A.~Wiles, \emph{On ordinary {$\lambda$}-adic representations associated to
  modular forms}, Invent. Math. \textbf{94} (1988), no.~3, 529--573.

\end{thebibliography}

\end{document}